\documentclass[oneside,11pt,reqno]{amsart}
\usepackage{amssymb,amsmath,amsthm,bbm,enumerate,mdwlist,url,multirow,hyperref}
\usepackage[pdftex]{graphicx}
\usepackage[shortlabels]{enumitem}

\theoremstyle{definition}
\newtheorem{definition}{Definition}%Extra square-bracket argument achieves that the numbering is the same as for definition (single uniform counter). 

\theoremstyle{theorem}
\newtheorem{proposition}[definition]{Proposition}

\newtheorem{theorem}[definition]{Theorem}

\newtheorem{corollary}[definition]{Corollary}
\theoremstyle{remark}
\newtheorem{remark}[definition]{Remark}
\numberwithin{equation}{section}
\numberwithin{definition}{section}
\def\PP{\mathsf P}
\def\ZZ{\mathbb Z}
\def\EE{\mathsf E}
\def\id{\mathrm{id}}

\def\RR{\mathbb R}

\def\PPm{\mathsf{P}^\natural}

\def\pr{\mathrm{pr}}
\def\Zh{\mathbb{Z}_h}
\def\sc{W}
\def\Zq{Z^{(q)}}
\def\Wq{W^{(q)}}

\def\Cr{\mathbb{C}^{\rightarrow}}
\def\Crr{\overline{\mathbb{C}^{\rightarrow}}}

\def\Cll{\overline{\mathbb{C}^{\leftarrow}}}
\def\Cd{\mathbb{C}^{\downarrow}}
\def\Cdd{\overline{\mathbb{C}^{\downarrow}}}

\def\Cuu{\overline{\mathbb{C}^{\uparrow}}}

\def\Exp{\mathrm{Exp}}
\def\geom{\mathrm{geom}}

\def\supp{\mathrm{supp}}

\def\measure{\lambda}

\begin{document}
\title{Fluctuation theory for upwards skip-free L\'evy chains}

\author{Matija Vidmar}
\address{Department of Mathematics, University of Ljubljana, Slovenia}
\email{matija.vidmar@fmf.uni-lj.si}

\thanks{The support of the Slovene Human Resources Development and Scholarship Fund under contract number 11010-543/2011 is acknowledged. I thank Andreas Kyprianou for suggesting to me some of the investigations in this paper.}

\begin{abstract}
A fluctuation theory and, in particular, a  theory of scale functions is developed for upwards skip-free L\'evy chains, i.e. for right-continuous random walks embedded into continuous time as compound Poisson processes. This is done by analogy to the spectrally negative class of L\'evy processes -- several results, however, can be made more explicit/exhaustive in our compound Poisson setting. In particular, the scale functions admit a linear recursion, of constant order when the support of the jump measure is bounded, by means of which they can be calculated -- some examples are considered.
\end{abstract}

\keywords{L\'evy processes, non-random overshoots, skip-free random walks, fluctuation theory, scale functions}

\subjclass[2010]{Primary: 60G51; Secondary: 60G50} 

\maketitle

\section{Introduction}\label{section:Introduction}
It was shown in \cite{vidmar} that precisely two types of L\'evy processes exhibit the property of non-random overshoots: those with no positive jumps a.s., and compound Poisson processes, whose jump chain is (for some $h>0$) a random walk on $\Zh:=\{hk\!\!:k\in\ZZ\}$, skip-free to the right. The latter class was then referred to as ``upwards skip-free L\'evy chains''. Also in the same paper it was remarked that this common property which the two classes share results in a more explicit fluctuation theory (including the Wiener-Hopf factorization) than for a general L\'evy process, this being rarely the case (cf. \cite [p. 172, Subsection 6.5.4]{kyprianou}).

Now, with reference to existing literature on fluctuation theory, the spectrally negative case (when there are no positive jumps, a.s.) is dealt with in detail in \cite[Chapter VII]{bertoin} \cite[Section 9.46]{sato} and especially \cite[Chapter 8]{kyprianou}. On the other hand no equally exhaustive treatment of the right-continuous random walk seems to have been presented thus far, but see \cite{quine,brown,marchal}  \cite[Section~4]{vylder} \cite[Section~7]{dickson} \cite[Section 9.3]{doney} \cite[\emph{passim}]{spitzer}. In particular, no such exposition appears forthcoming for the continuous-time analogue of such random walks, wherein the connection and analogy to the spectrally negative class of L\'evy processes becomes most transparent and direct. 

In the present paper we proceed to do just that, i.e. we develop, by analogy to the spectrally negative case, a complete fluctuation theory (including theory of scale functions) for upwards skip-free L\'evy chains. Indeed, the transposition of the results from the spectrally negative to the skip-free setting is mostly straightforward. Over and above this, however, and beyond what is purely analogous to the exposition of the spectrally negative case, (i) further specifics of the reflected process (Theorem~\ref{theorem:locals}\ref{USF:reflectedprocess}), of the excursions from the supremum (Theorem~\ref{theorem:locals}\ref{USF:excursion}) and of the inverse of the local time at the maximum (Theorem~\ref{theorem:locals}\ref{USF:localtime:further}) are identified, (ii) the class of subordinators that are the descending ladder heights processes of such upwards skip-free L\'evy chains is precisely characterized (Theorem~\ref{theorem:subordinator}), and (iii) a linear recursion is presented which allows us to directly compute the families of scale functions (Eq.~\eqref{equation:recursion:Wq}, \eqref{equation:recursion:Zq}, Proposition~\ref{proposition:calculation} and Corollary~\ref{proposition:calculating_scale_functions}).

Application-wise, note e.g. that the classical continuous-time Bienaym\'e-Galton-Watson branching process is associated to upwards skip-free L\'evy chains via a suitable time change \cite[Section~1.3.4]{kyprianou}. Upwards skip-free L\'evy chains  also feature as a natural continuous-time approximation of the more subtle spectrally negative family \cite{vidmar:scales}. 

The organisation of the rest of this paper is as follows. Section~\ref{section:Setting} introduces the setting and notation. Then Section~\ref{section:Fluctuation_theory} develops the relevant fluctuation theory, in particular details of the Wiener-Hopf factorization. Finally, Section~\ref{section:Theory_of_scale_functions} deals with the two-sided exit problem and the accompanying families of scale functions. 
%For the readers convenience, Appendix~\ref{section:the_reflected_process} connects the generator of a compound Poisson process $X$ living on $\Zh$ with that of the reflected process and Appendix~\ref{appendix:kolmogorov_consistency} discusses martingale changes of measures. 

\section{Setting and notation}\label{section:Setting}
Let $(\Omega,\mathcal{F},\mathbb{F}=(\mathcal{F}_t)_{t\geq 0},\PP)$ be a filtered probability space supporting a L\'evy process \cite[p. 2, Definition~1.1]{kyprianou} $X$ ($X$ is assumed to be $\mathbb{F}$-adapted and to have independent increments relative to $\mathbb{F}$). The L\'evy measure \cite[p. 38, Definition~8.2]{sato} of $X$ is denoted by $\measure$. Next, recall from \cite{vidmar} (with $\supp(\nu)$ denoting the support \cite[p. 9]{kallenberg} of a measure $\nu$ defined on the Borel $\sigma$-field of some topological space): 

\begin{definition}[Upwards skip-free L\'evy chain]
$X$ is an \emph{upwards skip-free L\'evy chain}, if it is a compound Poisson process \cite[p. 18, Definition~4.2]{sato}, and for some $h>0$, $\supp(\lambda)\subset \Zh$, whereas $\supp(\lambda\vert_{\mathcal{B}((0,\infty))})=\{h\}$. 
\end{definition}
In the sequel, $X$ will be assumed throughout an upwards skip-free L\'evy chain, with $\measure(\{h\})>0$ ($h>0$) and characteristic exponent $\Psi(p)=\int (e^{ipx}-1)\measure(dx)$ ($p\in\mathbb{R}$). In general, we insist on (i) every sample path of $X$ being c\`adl\`ag (i.e. right-continuous, admitting left limits) and (ii) $(\Omega,\mathcal{F},\mathbb{F},\PP)$ satisfying the standard assumptions  (i.e. the $\sigma$-field $\mathcal{F}$ is $\PP$-complete, the filtration $\mathbb{F}$ is right-continuous and $\mathcal{F}_0$ contains all $\PP$-null sets). Nevertheless, we shall, sometimes and then only provisionally, relax assumption (ii), by transferring $X$ as the coordinate process onto the canonical space $\mathbb{D}_h:=\{\omega\in \Zh^{[0,\infty)}:\omega\text{ is c\`adl\`ag}\}$ of c\`adl\`ag paths, mapping $[0,\infty)\to \Zh$, equipping $\mathbb{D}_h$ with the $\sigma$-algebra and natural filtration of evaluation maps; this, however, will always be made explicit. We allow $e_1$ to be exponentially distributed, mean one, and independent of $X$; then define $e_p:=e_1/p$ ($p\in (0,\infty)\backslash\{1\}$). 

Further, for $x\in\mathbb{R}$, introduce $T_x:=\inf \{t\geq 0:X_t\geq x\}$, the first entrance time of $X$ into $[x,\infty)$. Note that $T_x$ is an $\mathbb{F}$-stopping time \cite[p. 101, Theorem 6.7]{kallenberg}. The supremum or maximum (respectively infimum or minimum) process of $X$ is denoted  $\overline{X}_t:=\sup\{X_s:s\in [0,t]\}$ (respectively $\underline{X}_t:=\inf\{X_s:s\in [0,t]\}$) ($t\geq 0$). $\underline{X}_{\infty}:=\inf\{X_s:s\in [0,\infty)\}$ is the overall infimum. 

With regard to miscellaneous general notation we have: 
\begin{enumerate}
\item The nonnegative, nonpositive, positive and negative real numbers are denoted by $\mathbb{R}_+:=\{x\in\mathbb{R}:x\geq 0\}$, $\mathbb{R}_-:=\{x\in\mathbb{R}:x\leq 0\}$, $\mathbb{R}^+:= \mathbb{R}_+\backslash \{0\}$ and $\mathbb{R}^-:= \mathbb{R}_-\backslash \{0\}$, respectively. Then $\mathbb{Z}_+:=\mathbb{R}_+\cap\mathbb{Z}$, $\mathbb{Z}_-:=\mathbb{R}_-\cap\mathbb{Z}$, $\mathbb{Z}^+:=\mathbb{R}^+\cap\mathbb{Z}$ and $\mathbb{Z}^-:=\mathbb{R}^-\cap\mathbb{Z}$ are the nonnegative, nonpositive, positive and negative integers, respectively.
\item Similarly,  for $h>0$: $\Zh^+:=\Zh\cap\mathbb{R}_+$, $\Zh^{++}:=\Zh\cap\mathbb{R}^+$, $\Zh^-:=\Zh\cap \mathbb{R}_-$ and $\Zh^{--}:=\Zh\cap \mathbb{R} ^-$ are the apposite elements of $\Zh$. 
\item The following introduces notation for the relevant half-planes of $\mathbb{C}$; the arrow notation is meant to be suggestive of which half-plane is being considered: $\mathbb{C}^{\rightarrow}:=\{z\in\mathbb{C}:\Re z>0\}$, $\mathbb{C}^{\leftarrow}:=\{z\in\mathbb{C}:\Re z<0\}$, $\mathbb{C}^{\downarrow}:=\{z\in\mathbb{C}:\Im z<0\}$ and $\mathbb{C}^{\uparrow}:=\{z\in\mathbb{C}:\Im z>0\}$. $\Crr$, $\Cll$, $\Cdd$ and $\Cuu$ are then the respective closures of these sets. 
\item $\mathbb{N}=\{1,2,\ldots\}$ and $\mathbb{N}_0=\mathbb{N}\cup \{0\}$ are the positive and nonnegative integers, respectively. $\lceil x\rceil:=\inf\{k\in\mathbb{Z}: k\geq x\}$ ($x\in \mathbb{R}$) is the ceiling function. For $\{a,b\}\subset [-\infty,+\infty]$: $a\land b:=\min\{a,b\}$ and $a\lor b:=\max\{a,b\}$. 
\item The Laplace transform of a measure $\mu$ on $\mathbb{R}$, concentrated on $[0,\infty)$, is denoted $\hat{\mu}$: $\hat{\mu}(\beta)=\int_{[0,\infty)}e^{-\beta x}\mu(dx)$ (for all $\beta\geq 0$ such that this integral is finite). To a nondecreasing right-continuous function $F:\mathbb{R}\to\mathbb{R}$, a measure $dF$ may be associated in the Lebesgue-Stieltjes sense. 
\end{enumerate}
The geometric law $\geom(p)$ with success parameter $p\in (0,1]$ has $\geom(p)(\{k\})=p(1-p)^k$ ($k\in\mathbb{N}_0$), $1-p$ is then the failure parameter. The exponential law $\Exp(\beta)$ with parameter $\beta>0$ is specified by the density $\Exp(\beta)(dt)=\beta e^{-\beta t}\mathbbm{1}_{(0,\infty)}(t)dt$. A function $f:[0,\infty)\to [0,\infty)$ is said to be of exponential order, if there are $\{\alpha,A\}\subset \mathbb{R}_+$, such that $f(x)\leq Ae^{\alpha x}$ ($x\geq 0$); $f(+\infty):=\lim_{x\to\infty}f(x)$, when this limit exists. DCT (respectively MCT) stands for the dominated (respectively monotone) convergence theorem. Finally, increasing (respectively decreasing) will mean strictly increasing (respectively strictly decreasing), nondecreasing (respectively nonincreasing) being used for the weaker alternative; we will understand $a/0=\pm \infty$ for $a\in \pm(0,\infty)$. 

\section{Fluctuation theory}\label{section:Fluctuation_theory}
In the following, to fully appreciate the similarity (and eventual differences) with the spectrally negative case, the reader is invited to directly compare the exposition of this subsection with that of \cite[Section VII.1]{bertoin} and \cite[Section 8.1]{kyprianou}. 

\subsection{Laplace exponent, the reflected process, local times and excursions from the supremum, supremum process and long-term behaviour, exponential change of measure}
Since the Poisson process admits exponential moments of all orders, it follows that $\EE[e^{\beta \overline{X}_t}]<\infty$ and, in particular, $\EE[e^{\beta X_t}]<\infty$ for all $\{\beta,t\}\subset [0,\infty)$. Indeed, it may be seen by a direct computation that for $\beta\in \Crr$, $t\geq 0$, $\EE[e^{\beta X_t}]=\exp\{t\psi(\beta)\}$, where $\psi(\beta):=\int_{\mathbb{R}}(e^{\beta x}-1)\lambda(dx)$ is the Laplace exponent of $X$. Moreover, $\psi$ is continuous (by the DCT) on $\Crr$ and analytic in $\Cr$ (use the theorems of Cauchy \cite[p. 206, 10.13 Cauchy's theorem for triangle]{rudin}, Morera \cite[p. 209, 10.17 Morera's theorem]{rudin} and Fubini).

%The characteristic function $(\beta\mapsto E[e^{i\beta X_1}]): \mathbb{R}\to \mathbb{C}$ can therefore be extended continuously (respectively analytically) onto $\Cdd$ (respectively $\Cd$), by the DCT (respectively the theorems of Cauchy, Morera and Fubini). On the other hand, the L\'evy-Khintchine formula and the fact that the L\'evy measure is concentrated on $\Zh^-\cup \{h\}$ show that the characteristic exponent $\Psi$ of $X$ is also well-defined, continuous on $\Cdd$ (DCT) and analytic on $\Cd$ (Cauchy, Morera and Fubini). Theorem~\ref{theorem:on_zeros_of_holomrphic_functions} then allows to conclude (for $\beta\in \Crr$, $t\geq 0$): 
%\begin{equation}\label{eq:skip-free:laplace}
%\EE[e^{\beta X_t}]=\exp\{t\psi(\beta)\},
%\end{equation}
%where $\psi(\beta):=\Psi(-i\beta)=\int_{\mathbb{R}}(e^{\beta x}-1)\lambda(dx)$ is the Laplace exponent of $X$. Indeed, fix $t\geq 0$ and let $f^1$ (respectively $f^2$) be the left (respectively right) hand-side of \eqref{eq:skip-free:laplace}. The difference $f^1-f^2$ vanishes on the imaginary axis, is continuous on $\Crr$ and analytic on $\Cr$. By pasting another mirrored copy of it onto $\Cll$, we have given a function which is analytic on the entire plane, by Morera's theorem and continuity. It vanishes on the imaginary axis and so vanishes identically.

Next, note that $\psi(\beta)$ tends to $+\infty$ as $\beta\to\infty$ over the reals, due to the presence of the atom of $\lambda$ at $h$. Upon restriction to $[0,\infty)$, $\psi$ is strictly convex, as follows first on $(0,\infty)$ by using differentiation under the integral sign and noting that the second derivative is strictly positive, and then extends to $[0,\infty)$ by continuity. 

Denote then by $\Phi(0)$ the largest root of $\psi\vert_{[0,\infty)}$. Indeed, $0$ is always a root, and due to strict convexity, if $\Phi(0)>0$, then $0$ and $\Phi(0)$ are the only two roots. The two cases occur, according as to whether $\psi'(0+)\geq 0$ or $\psi'(0+)<0$, which is clear. It is less obvious, but nevertheless true, that this right derivative at $0$ actually exists, indeed $\psi'(0+)=\int_\mathbb{R}x\lambda(dx)\in [-\infty,\infty)$. This follows from the fact that $(e^{\beta x}-1)/\beta$ is nonincreasing as $\beta\downarrow 0$ for $x\in \mathbb{R}_-$ and hence monotone convergence applies. Continuing from this, and with a similar justification, one also gets the equality $\psi''(0+)=\int x^2\lambda(dx)\in (0,+\infty]$ (where we agree $\psi''(0+)=+\infty$ if $\psi'(0+)=-\infty$). In any case, $\psi:[\Phi(0),\infty)\to [0,\infty)$ is continuous and increasing, it is a bijection and we let $\Phi:[0,\infty)\to [\Phi(0),\infty)$ be the inverse bijection, so that $\psi\circ \Phi=\id_\mathbb{R_+}$.

With these preliminaries having been established, our first theorem identifies characteristics of the reflected process, the local time of $X$ at the maximum (for a definition of which see e.g. \cite[p. 140, Definition 6.1]{kyprianou}), its inverse, as well as the expected length of excursions and the probability of an infinite excursion therefrom (for definitions of these terms see e.g. \cite[pp. 140-147]{kyprianou}; we agree that an excursion (from the maximum) starts immediately $X$ leaves its running maximum and ends immediately it returns to it; by its length we mean the amount of time between these two time points).

\begin{theorem}[Reflected process; (inverse) local time; excursions]\label{theorem:locals} Let $q_n:=\measure(\{-nh\})/\measure(\mathbb{R})$ for $n\in \mathbb{N}$ and $p:=\measure(\{h\})/\measure(\mathbb{R})$.
\leavevmode
\begin{enumerate}[(i)]
\item\label{USF:reflectedprocess} The generator matrix $\tilde{Q}$ of the Markov process $Y:=\overline{X}-X$ on $\Zh^+$ is given by (with $\{s,s'\}\subset \Zh^+$): $\tilde{Q}_{ss'}=\measure(\{s-s'\})-\delta_{ss'}\measure(\mathbb{R})$, unless $s=s'=0$, in which case we have $\tilde{Q}_{ss'}=-\measure((-\infty,0))$.
\item\label{USF:localtime} For the reflected process $Y$, $0$ is a holding point. The actual time spent at $0$ by $Y$, which we shall denote $L$, is a local time at the maximum. Its right-continuous inverse $L^{-1}$, given by  $L^{-1}_t:=\inf\{s\geq 0:L_s>t\}$ (for $0\leq t<L_{\infty}$; $L_t^{-1}:=\infty$ otherwise), is then a (possibly killed) compound Poisson subordinator with unit positive drift. 
\item\label{USF:excursion} Assuming that $\lambda((-\infty,0))>0$ to avoid the trivial case, the expected length of an excursion away from the supremum is equal to $\frac{\lambda(\{h\})h-\psi'(0+)}{(\psi'(0+)\lor 0)\lambda((-\infty,0))}$; whereas the probability of such an excursion being infinite is $\frac{\measure(\{h\})}{\measure((-\infty,0))}(e^{\Phi(0)h}-1)=:p^*$. 
\item\label{USF:localtime:further} Assume again $\lambda((-\infty,0))>0$ to avoid the trivial case. Let $N$, taking values in $\mathbb{N}\cup \{+\infty\}$, be the number of jumps the chain makes before returning to its running maximum, after it has first left it (it does so with probability $1$). Then the law of $L^{-1}$ is given by (for $\theta\in [0,+\infty)$):
$$-\log \EE \left[\exp(-\theta L^{-1}_1)\mathbbm{1}_{\{L^{-1}_1<+\infty\}}\right]=\theta +\lambda((-\infty,0))\left( 1-\sum_{k=1}^\infty\PP(N=k)\left(\frac{\lambda(\mathbb{R})}{\lambda(\mathbb{R})+\theta}\right)^k\right).$$ In particular, $L^{-1}$ has a killing rate of $\lambda((-\infty,0))p^*$, L\'evy mass $\lambda((-\infty,0))(1-p^*)$ and its jumps have the probability law on $(0,+\infty)$ given by the length of a generic excursion from the supremum, conditional on it being finite, i.e. that of an independent $N$-fold sum of independent $\Exp(\lambda(\mathbb{R}))$-distributed random variables, conditional on $N$ being finite. Moreover, one has, for $k\in \mathbb{N}$, $\PP(N=k)=\sum_{l=1}^kq_lp_{l,k}$, where the coefficients $(p_{l,k})_{l,k=1}^\infty$ satisfy the initial conditions:
\begin{equation*}
p_{l,1}=p\delta_{l1},\quad l\in \mathbb{N};
\end{equation*}
the recursions:
\begin{equation*}
p_{l,k+1}=
\begin{cases}
0 & \text{if }l=k\text{ or }l>k+1\\
\sum_{m=1}^{k-1}q_mp_{m+1,k}&\text{if }l=1\\
p^{k+1}& \text{if }l=k+1\\
pp_{l-1,k}+\sum_{m=1}^{k-l}q_mp_{m+l,k}& \text{if }1<l<k
\end{cases},\qquad \{l,k\}\subset \mathbb{N};
\end{equation*}
and $p_{l,k}$ may be interpreted as the probability of $X$ reaching level $0$ starting from level $-lh$ for the first time on precisely the $k$-th jump ($\{l,k\}\subset \mathbb{N}$). 
\end{enumerate}
\end{theorem}
%\begin{remark}
%To the best of the author's knowledge, parts \ref{USF:reflectedprocess} and \ref{USF:excursion} represent a level of explicitness which has not (yet) been obtained in the general spectrally negative setting. 
%\end{remark}
\begin{proof}
\ref{USF:reflectedprocess} is clear, since, e.g. $Y$ transitions away from $0$ at the rate at which $X$ makes a negative jump; and from $s\in \Zh^{+}\backslash \{0\}$ to $0$ at the rate at which $X$ jumps up by $s$ or more etc. 

\ref{USF:localtime} is standard \cite[p. 141, Example~6.3 \& p. 149, Theorem~6.10]{kyprianou}.
%The (c\`adl\`ag) inverse local time drifts while $X$ is at its running maximum and it jumps at the levels corresponding to each excursion away from the maximum, by the length of that excursion. The jump is to $+\infty$, if it corresponds to an infinite excursion away from the supremum. Moreover, the process $S$ regenerates every time $X$ visits its current maximum (i.e. $S$ has completed a jump) and therefore $S$ has jumps occurring at independent, $\Exp(\lambda(-\infty,0))$-distributed, time intervals, up to the explosion time. Next note that there is always a fixed probability that the excursion will be infinite, whereas the successive amounts of time $X$ spends at the maximum are iid random variables, independent of the excursion lengths. It follows that $S$ is indeed killed at an independent exponential random time, since a geometric independent sum of iid exponential random variables (amounts of time spent at the maximum) is in turn an exponential random variable \cite[p. 54]{feller2}.
%\footnote{This is a simple exercise in conditioning and characteristic functions. The argument which follows also makes the formulation of the statement precise. So let $N\sim \geom(p)$ be independent of the independency $(E_i)_{i\geq 1}$ of $\Exp(\lambda)$ variables. Define $Z:=\sum_{i=1}^{N+1}E_i$. Then we have for all real $\theta$, $\EE[e^{i\theta Z}]=\sum_{k=0}^\infty p(1-p)^{k}(1/(1-i\theta/\lambda))^{k+1}=\frac{p}{1-i\theta/\lambda} \frac{1}{1-\frac{1-p}{1-i\theta/\lambda}}=1/(1-(i\theta)/(p \lambda))$.} 

We next establish \ref{USF:excursion}.  Denote, provisionally, by $\beta$ the expected excursion length. Further, let the discrete-time Markov chain $W$ (on the state space $\mathbb{N}_0$) be endowed with the initial distribution $w_j:=\frac{q_j}{1-p}$ for $j\in\mathbb{N}$, $w_0:=0$; and transition matrix $P$, given by $P_{0i}=\delta_{0i}$, whereas for $i\geq 1$: $P_{ij}=p$, if $j=i-1$; $P_{ij}=q_{j-i}$, if $j>i$; and $P_{ij}=0$ otherwise ($W$ jumps down with probability $p$, up $i$ steps with probability $q_i$, $i\geq1$, until it reaches $0$, where it gets stuck). Let further $N$ be the first hitting time for $W$ of $\{0\}$, so that a typical excursion length of $X$ is equal in distribution to an independent sum of $N$ (possibly infinite) $\Exp(\measure(\mathbb{R}))$-random variables. It is Wald's identity that $\beta=(1/\measure(\mathbb{R}))\EE[N]$. Then (in the obvious notation, where $\underline{\infty}$ indicates the sum is inclusive of $\infty$), by Fubini: $\EE[N]=\sum_{n=1}^{\underline{\infty}} n \sum_{l=1}^\infty w_l\PP_l(N=n)=\sum_{l=1}^\infty w_lk_l$, where $k_l$ is the mean hitting time of $\{0\}$ for $W$, if it starts from $l\in\mathbb{N}_0$, as in \cite[p. 12]{norris}. From the skip-free property of the chain $W$ it is moreover transparent that $k_i=\alpha i$, $i\in \mathbb{N}_0$, for some $0<\alpha\leq \infty$ (with the usual convention $0\cdot\infty=0$). 
%Indeed starting from $i$, the chain must first get to $i-1$, for which it uses on average an $\alpha$ amount of time, and then must repeat this procedure a total of $i$ times, each repeat being independent of the previous ones, and equal in distribution. 
Moreover we know \cite[p. 17, Theorem 1.3.5]{norris} that $(k_i:i\in\mathbb{N}_0)$ is the minimal solution to $k_0=0$ and $k_i=1+\sum_{j=1}^\infty P_{ij}k_j$ ($i\in\mathbb{N}$). Plugging in $k_i=\alpha i$, the last system of linear equations is equivalent to (provided $\alpha<\infty$) $0=1-p\alpha+\alpha\zeta$, where $\zeta:=\sum_{j=1}^\infty jq_j$. Thus, if $\zeta<p$, the minimal solution to the system is $k_i=i/(p-\zeta)$, $i\in \mathbb{N}_0$, from which $\beta=\zeta/(\measure((-\infty,0)) (p-\zeta))$ follows at once. If $\zeta\geq p$, clearly we must have $\alpha=+\infty$, hence $\EE[N]=+\infty$ and thus $\beta=+\infty$.

To establish the probability of an excursion being infinite, i.e. $\sum_{i=1}^\infty q_i(1-\alpha_i)/\sum_{i=1}^\infty q_i$, where $\alpha_i:=\PP_i(N<\infty)>0$, we see that (by the skip-free property) $\alpha_i=\alpha_1^i$, $i\in \mathbb{N}_0$, and by the strong Markov property, for $i\in \mathbb{N}$, $\alpha_i=p\alpha_{i-1}+\sum_{j=1}^\infty q_{j}\alpha_{i+j}$. It follows that $1=p\alpha_1^{-1}+\sum_{j=1}^\infty q_j\alpha_1^j$, i.e. $0=\psi(\log(\alpha_1^{-1})/h)$. Hence, by Theorem~\ref{corollary:longtimeandsup}\ref{ref:longtimebehaviour}, whose proof will be independent of this one, $\alpha_1=e^{-\Phi(0)h}$ (since $\alpha_1<1$, if and only if $X$ drifts to $-\infty$). 

Finally, \ref{USF:localtime:further} is straightforward. 
\end{proof}
We turn our attention now to the supremum process $\overline{X}$. First, using the lack of memory property of the exponential law and the skip-free nature of $X$, we deduce from the strong Markov property applied at the time $T_a$, that for every $a,b\in \Zh^{+}$, $p>0$: $\PP(T_{a+b}<e_p)=\PP(T_a<e_p)\PP(T_b<e_p).$ In particular, for any $n\in \mathbb{N}_0$: $\PP(T_{nh}<e_p)=\PP(T_h<e_p)^n.$ And since for $s\in \Zh^{+}$, $\{T_{s}<e_p\}=\{\overline{X}_{e_p}\geq s\}$ ($\PP$-a.s.)
%--- taking into account that $\PP$-a.s. $e_p$ never equals a jump time of $X$ (use independence, denumerability of jumps, absolute continuity of distributions and the fact that the diagonal in $\mathbb{R}_+^2$ has zero Lebesgue measure) --- 
one has (for $n\in \mathbb{N}_0$): $\PP(\overline{X}_{e_p}\geq nh)=\PP(\overline{X}_{e_p}\geq h)^n$. Therefore $\overline{X}_{e_p}/h\sim\geom (1-\PP(\overline{X}_{e_p}\geq h))$. 

Next, to identify $\PP(\overline{X}_{e_p}\geq h)$, $p>0$, observe that (for $\beta\geq 0$, $t\geq 0$): $\EE[\exp\{\Phi(\beta)X_t\}]=e^{t\beta}$ and hence $(\exp\{\Phi(\beta)X_t-\beta t\})_{t\geq 0}$ is an $(\mathbb{F},\PP)$-martingale by stationary independent increments of $X$, for each $\beta\geq 0$. Then apply the Optional Sampling Theorem at the bounded stopping time $T_x\land t$ ($t,x\geq 0$) to get: $$\EE[\exp\{\Phi(\beta)X(T_x\land t)-\beta(T_x\land t)\}]=1.$$ Note that $X(T_x\land t)\leq h\lceil x/h\rceil$ and $\Phi(\beta)X(T_x\land t)-\beta (T_x\land t)$ converges to $\Phi(\beta)h\lceil x/h\rceil-\beta T_x$ ($\PP$-a.s.) as $t\to\infty$ on $\{T_x<\infty\}$. It converges to $-\infty$ on the complement of this event, $\PP$-a.s., provided $\beta+\Phi(\beta)>0$. Therefore we deduce by dominated convergence, first for $\beta>0$ and then also for $\beta=0$, by taking limits: 
\begin{equation}\label{eq:laplace_exponent_for_T_x}
\EE[\exp\{-\beta T_x\}\mathbbm{1}_{\{T_x<\infty\}}]=\exp\{-\Phi(\beta)h\lceil x/h\rceil\}.
\end{equation}
%The latter allows to establish $\PP(\overline{X}_{e_p}\geq h)$ for $p>0$ and then to characterize the long-term behaviour of $X$ in terms of $\psi'(0+)$. 

Before we formulate out next theorem, recall also that any non-zero L\'evy process either drifts to $+\infty$, oscillates or drifts to $-\infty$ \cite[pp. 255-256, Proposition~37.10 and Definition~37.11]{sato}.

\begin{theorem}[Supremum process and long-term behaviour]\label{corollary:longtimeandsup}
\leavevmode
\begin{enumerate}[(i)]
\item\label{ref:supremumprocess}  The failure probability for the geometrically distributed $\overline{X}_{e_p}/h$ is $\exp\{-\Phi(p)h\}$ ($p>0$). 
\item\label{ref:longtimebehaviour} $X$ drifts to $+\infty$, oscillates or drifts to $-\infty$ according as to whether $\psi'(0+)$ is positive, zero, or negative. In the latter case $\overline{X}_\infty/h$ has a geometric distribution with failure probability $\exp\{-\Phi(0)h\}$. 
\item\label{ref:newmaxima} $(T_{nh})_{n\in\mathbb{N}_0}$ is a discrete-time increasing stochastic process, vanishing at $0$ and having stationary independent increments up to the explosion time, which is an independent geometric random variable; it is a killed random walk. 
\end{enumerate}
\end{theorem}
\begin{remark}
Unlike in the spectrally negative case \cite[p. 189]{bertoin}, the supremum process cannot be obtained from the reflected process, since the latter does not discern a point of increase in $X$ when the latter is at its running maximum.
\end{remark}
\begin{proof}
We have for every $s\in \Zh^+$: 
\begin{equation}\label{eq:supremum_independent_exp}
\PP(\overline{X}_{e_p}\geq s)=\PP(T_{s}<e_p)=\EE[\exp\{-pT_{s}\}\mathbbm{1}_{\{T_s<\infty\}}]=\exp\{-\Phi(p)s\}.
\end{equation}
Thus \ref{ref:supremumprocess} obtains. 

For \ref{ref:longtimebehaviour} note that letting $p\downarrow 0$ in \eqref{eq:supremum_independent_exp}, we obtain $\overline{X}_\infty<\infty$ ($\PP$-a.s.), if and only if $\Phi(0)>0$, which is equivalent to $\psi'(0+)<0$. If so, $\overline{X}_\infty/h$ is geometrically distributed with failure probability $\exp\{-\Phi(0)h\}$ and then (and only then) does $X$ drift to $-\infty$.

It remains to consider drifting to $+\infty$ (the cases being mutually exclusive and exhaustive). Indeed, $X$ drifts to $+\infty$, if and only if $\EE[T_s]$ is finite for each $s\in \Zh^+$ \cite[p. 172, Proposition VI.17]{bertoin}. Using again the nondecreasingness of $(e^{-\beta T_s}-1)/\beta$ in $\beta\in [0,\infty)$, we deduce from \eqref{eq:laplace_exponent_for_T_x}, by monotone convergence, that one may differentiate under the integral sign, to get $\EE[T_s\mathbbm{1}_{\{T_s<\infty\}}]=(\beta\mapsto -\exp\{-\Phi(\beta)s\})'(0+)$. So the $\EE[T_s]$ are finite, if and only if $\Phi(0)=0$ (so that $T_s<\infty$ $\PP$-a.s.) and $\Phi'(0+)<\infty$. Since $\Phi$ is the inverse of $\psi\vert_{[\Phi(0),\infty)}$, this is equivalent to saying $\psi'(0+)>0$. 

Finally, \ref{ref:newmaxima} is clear.
\end{proof}

\begin{table}[!hbt]
\caption{Connections between the quantities $\psi'(0+)$, $\Phi(0)$, $\Phi'(0+)$. Behaviour of $X$ at large times and of its excursions away from the running supremum (the latter if $\lambda((-\infty,0))>0$).
% --- only then do we have such excursions (with a positive probability)).
}\label{table:longtimeandexcursions}
\begin{center}
\begin{tabular}{|c|c|c|c|c|}\hline
$\psi'(0+)$ & $\Phi(0)$ & $\Phi'(0+)$ & Long-term behaviour & Excursion length\\\hline
$\in (0,\infty)$ & $0$& $\in (0,\infty)$ & drifts to $+\infty$ & finite expectation\\\hline
$0$ & $0$& $+\infty$ & oscillates & a.s. finite with infinite expectation\\\hline
$\in [-\infty,0)$ & $\in (0,\infty)$& $\in (0,\infty)$ & drifts to $-\infty$ & infinite with a positive probability\\\hline
\end{tabular}
\end{center}
\end{table}
We conclude this section by offering a way to reduce the general case of an upwards skip-free L\'evy chain to one which necessarily drifts to $+\infty$. This will prove useful in the sequel. First, there is a pathwise approximation of an oscillating $X$, by (what is again) an upwards skip-free L\'evy chain, but drifting to infinity.

\begin{remark}\label{remark:approximate_oscillating_case}
Suppose $X$ oscillates. Let (possibly by enlarging the probability space to accommodate for it) $N$ be an independent Poisson process with intensity $1$ and $N^\epsilon_t:=N_{t\epsilon}$ ($t\geq 0$) so that $N^\epsilon$ is a Poisson process with intensity $\epsilon$, independent of $X$. Define $X^\epsilon:=X+hN^\epsilon$. Then, as $\epsilon\downarrow 0$,  $X^\epsilon$ converges to $X$, uniformly on bounded time sets, almost surely, and is clearly an upwards skip-free L\'evy chain drifting to $+\infty$.
\end{remark}
The reduction of the case when $X$ drifts to $-\infty$ is somewhat more involved and is done by a change of measure. For this purpose assume until the end of this subsection, that $X$ is already the coordinate process on the canonical space $\Omega=\mathbb{D}_h$, equipped with the $\sigma$-algebra $\mathcal{F}$ and filtration $\mathbb{F}$ of evaluation maps (so that $\PP$ \emph{coincides} with the law of $X$ on $\mathbb{D}_h$ and $\mathcal{F}=\sigma(\pr_s:s\in [0,+\infty))$, whilst for $t\geq 0$, $\mathcal{F}_t=\sigma(\pr_s:s\in [0,t])$, where $\pr_s(\omega)=\omega(s)$, for $(s,\omega)\in [0,+\infty)\times \mathbb{D}_h$). We make this transition in order to be able to apply the Kolmogorov extension theorem in the proposition, which follows. Note, however, that we are no longer able to assume standard conditions on $(\Omega,\mathcal{F},\mathbb{F},\PP)$. Notwithstanding this, $(T_x)_{x\in\mathbb{R}}$ remain $\mathbb{F}$-stopping times, since by the nature of the space $\mathbb{D}_h$, for $x\in\mathbb{R}$, $t\geq 0$, $\{T_x\leq t\}=\{\overline{X}_t\geq x\}\in\mathcal{F}_t$.

\begin{proposition}[Exponential change of measure]\label{proposition:exponential_change_of_measure}
Let $c\geq 0$. Then, demanding: 
\begin{equation}\label{eq:exponential_change}
\PP_c(\Lambda)=\EE[\exp \{cX_t-\psi(c)t\}\mathbbm{1}_\Lambda]\hspace{0.5cm}(\Lambda\in\mathcal{F}_t,t\geq 0)
\end{equation}
this introduces a unique measure $\PP_c$ on $\mathcal{F}$. Under the new measure, $X$ remains an upwards skip-free L\'evy chain with Laplace exponent $\psi_c=\psi(\cdot+c)-\psi(c)$, drifting to $+\infty$, if $c\geq \Phi(0)$, unless $c=\psi'(0+)=0$. Moreover, if $\lambda_c$ is the new L\'evy measure of $X$ under $\PP_c$, then $\lambda_c\ll\lambda$ and $\frac{d\lambda_c}{d\lambda}(x)=e^{c x}$ $\lambda$-a.e. in $x\in\mathbb{R}$. Finally, for every $\mathbb{F}$-stopping time $T$, $\PP_c\ll \PP$ on restriction to $\mathcal{F}_T':=\{A\cap \{T<\infty\}:A\in\mathcal{F}_T\}$, and: $$\frac{d\PP_c\vert_{\mathcal{F}_T'}}{d\PP\vert_{\mathcal{F}_T'}}=\exp \{cX_T-\psi(c)T\}.$$
\end{proposition}
\begin{proof}
That $\PP_c$ is introduced consistently as a probability measure on $\mathcal{F}$ follows from the Kolmogorov extension theorem \cite[p. 143, Theorem 4.2]{parthasarathy}. Indeed, $M:=(\exp \{cX_t-\psi(c)t\})_{t\geq 0}$ is a nonnegative martingale (use independence and stationarity of increments of $X$ and the definition of the Laplace exponent), equal identically to $1$ at time $0$. 

Further, for all $\beta\in\Crr$, $\{t,s\}\subset \mathbb{R}_+$ and $\Lambda\in\mathcal{F}_t$: 
\begin{eqnarray*}
\EE_c[\exp\{\beta(X_{t+s}-X_t)\}\mathbbm{1}_\Lambda]&=&\EE[\exp\{cX_{t+s}-\psi(c)(t+s)\}\exp\{\beta(X_{t+s}-X_t)\}\mathbbm{1}_\Lambda]\\
&=&\EE[\exp\{(c+\beta)(X_{t+s}-X_t)-\psi(c)s\}]\EE[\exp\{cX_t-\psi(c)t\}\mathbbm{1}_\Lambda]\\
&=&\exp\{s(\psi(c+\beta)-\psi(c))\}\PP_c(\Lambda).
\end{eqnarray*}
An application of the Functional Monotone Class Theorem then shows that $X$ is indeed a L\'evy process on $(\Omega,\mathcal{F},\mathbb{F},\PP_c)$ and its Laplace exponent under $\PP_c$ is as stipulated (that $X_0=0$ $\PP_c$-a.s. follows from the absolute continuity of $\PP_c$ with respect to $\PP$ on restriction to $\mathcal{F}_0$). 

Next, from the expression for $\psi_c$, the claim regarding $\measure_c$ follows at once. Then clearly $X$ remains an upwards skip-free L\'evy chain under $\PP_c$, drifting to $+\infty$, if $\psi'(c+)>0$. 

Finally, let $A\in\mathcal{F}_T$ and $t\geq 0$. Then $A\cap \{T\leq t\}\in\mathcal{F}_{T\land t}$, and by the Optional Sampling Theorem:
\begin{equation*}
\PP_c(A\cap \{T\leq t\})=\EE[M_t\mathbbm{1}_{A\cap \{T\leq t\}}]=\EE[\EE[M_t\mathbbm{1}_{A\cap \{T\leq t\}}\vert \mathcal{F}_{T\land t}]]=\EE[M_{T\land t} \mathbbm{1}_{A\cap \{T\leq t\}}]=\EE[M_{T} \mathbbm{1}_{A\cap \{T\leq t\}}].
\end{equation*}
Using the MCT, letting $t\to\infty$, we obtain the equality $\PP_c(A\cap \{T<\infty\})=\EE[M_T\mathbbm{1}_{A\cap \{T<\infty\}}]$.
\end{proof}

\begin{proposition}[Conditioning to drift to $+\infty$]\label{proposition:conditioned_to_drift_to_infinity}
Assume $\Phi(0)>0$ and denote $\PPm:=\PP_{\Phi(0)}$ (see \eqref{eq:exponential_change}). We then have as follows.
\begin{enumerate}[(1)]
\item\label{lemma:drifttoinfty:one} For every $\Lambda\in\mathcal{A}:=\cup_{t\geq 0}\mathcal{F}_t$, $\lim_{n\to\infty}\PP(\Lambda\vert \overline{X}_\infty\geq nh)=\PPm(\Lambda)$.
\item\label{lemma:drifttoinfty:two} For every $x\geq 0$, the stopped process $X^{T_x}=(X_{t\land T_x})_{t\geq 0}$ is identical in law under the measures $\PPm$ and $\PP(\cdot\vert T_x<\infty)$ on the canonical space $\mathbb{D}_h$. 
\end{enumerate}
\end{proposition}
\begin{proof}
With regard to \ref{lemma:drifttoinfty:one}, we have as follows. Let $t\geq 0$. By the Markov property of $X$ at time $t$,  the process $\overset{\triangle}{X}:=(X_{t+s}-X_t)_{s\geq 0}$ is identical in law with $X$ on $\mathbb{D}_h$ and independent of $\mathcal{F}_t$ under $\PP$. Thus, letting $\overset{\triangle}{T}_y:=\inf\{t\geq 0:\overset{\triangle}{X}_t\geq y\}$ ($y\in\mathbb{R}$), one has for $\Lambda\in\mathcal{F}_t$ and $n\in\mathbb{N}_0$, by conditioning:
\begin{equation*}
\PP(\Lambda\cap \{t<T_{nh}<\infty\})=\EE[\EE[\mathbbm{1}_\Lambda\mathbbm{1}_{\{t<T_{nh}\}}\mathbbm{1}_{\{\overset{\triangle}{T}_{nh-X_t}<\infty\}}\vert\mathcal{F}_t]]=\EE[e^{\Phi(0)(X_t-nh)}\mathbbm{1}_{\Lambda\cap \{t<T_{nh}\}}],
\end{equation*}
since $\{\Lambda,\{t<T_{nh}\}\}\cup \sigma(X_t)\subset\mathcal{F}_t$. Next, noting that $\{\overline{X}_\infty \geq nh\}=\{T_{nh}<\infty\}$:
\begin{eqnarray*}
\PP(\Lambda\vert \overline{X}_\infty>nh)&=&e^{\Phi(0)nh}\left(\PP(\Lambda\cap \{T_{nh}\leq t\})+\PP(\Lambda\cap \{t<T_{nh}<\infty\})\right)\\
&=&e^{\Phi(0)nh}\left(\PP(\Lambda\cap\{T_{nh}\leq t\})+\EE[e^{\Phi(0)(X_t-nh)}\mathbbm{1}_{\Lambda\cap \{t<T_{nh}\}}]\right)\\
&=&e^{\Phi(0)nh}\PP(\Lambda\cap \{T_{nh}\leq t\})+\PPm(\Lambda\cap \{t<T_{nh}\}).
\end{eqnarray*}
The second term clearly converges to $\PPm(\Lambda)$ as $n\to\infty$. The first converges to $0$, because by \eqref{eq:supremum_independent_exp} $\PP(\overline{X}_{e_1}\geq nh)=e^{-nh\Phi(1)}=o(e^{-nh\Phi(0)})$, as $n\to\infty$, and we have the estimate $\PP(T_{nh}\leq t)=\PP(\overline{X}_t\geq nh)=\PP(\overline{X}_t\geq nh\vert e_1\geq t)\leq \PP(\overline{X}_{e_1}\geq nh\vert e_1\geq t)\leq e^t \PP(\overline{X}_{e_1}\geq nh)$. 

We next show \ref{lemma:drifttoinfty:two}. Note first that $X$ is $\mathbb{F}$-progressively measurable (in particular, measurable), hence the stopped process $X^{T_x}$ is measurable as a mapping into $\mathbb{D}_h$ \cite[p. 5, Problem~1.16]{karatzas}. 

Further, by the strong Markov property, conditionally on $\{T_x<\infty\}$, $\mathcal{F}_{T_x}$ is independent of the future increments of $X$ after $T_x$, hence also of $\{T_{x'}<\infty\}$ for any $x'>x$. We deduce that the law of $X^{T_x}$ is the same under $\PP(\cdot\vert T_x<\infty)$ as it is under $\PP(\cdot\vert T_{x'}<\infty)$ for any $x'>x$. \ref{lemma:drifttoinfty:two} then follows from \ref{lemma:drifttoinfty:one} by letting $x'$ tend to $+\infty$, the algebra $\mathcal{A}$ being sufficient to determine equality in law by a $\pi$/$\lambda$-argument.
\end{proof}
%As a final point, let us remark that making some assumption on the underlying filtered probability space for the local change of measure to work, was essential. Indeed, take any upwards skip-free L\'evy chain drifting to $-\infty$ (such L\'evy processes of course exist). Equip the space $D$ of c\`adl\`ag paths in $\mathbb{R}^{[0,\infty)}$, which vanish at $0$ and have a bounded overall supremum with the canonical filtration generated by the evaluation maps and the terminal $\sigma$-field. Possibly by discarding a $\PP$-negligible set, one has the mapping $(\omega\mapsto X(\omega,\cdot)):\Omega\to D$, which is measurable. Let $P$ be the image measure of $\PP$ under this mapping. Then the coordinate process is an upwards skip-free L\'evy chain drifting to $-\infty$ on the filtered space $D$, under the measure $P$, equivalent in law to $X$ under $\PP$. Now, we could also augment the filtration to get a standard filtered probability space, if we so wished. Regardless, clearly we would not be able to recover $\PPm$ by the above procedure, since this would necessitate $\PPm$ assigning all its mass to a set of paths drifting to $+\infty$, of which there are none in $D$!

\subsection{Wiener-Hopf factorization}
%In order to proceed to the Wiener-Hopf factorization, we first recall this result for a general compound Poisson process, thereby also introducing the relevant notation. We then particularize to our setting.
%
%\begin{remark}
%While this general form of the Wiener-Hopf factorization will, on account of notation, be stated for the upwards skip-free L\'evy chain $X$, Proposition~\ref{proposition:WHcompoundPoissonprocess} would, in fact, hold equally true of any compound Poisson process. 
%\end{remark}

\begin{definition}\label{definition:WHcompoundPoissonprocess}
We define, for $t\geq 0$, $\overline{G}^*_t:=\inf \{s\in [0,t]:X_s=\overline{X}_t\}$, i.e., $\PP$-a.s., $\overline{G}^*_t$ is the last time in the interval $[0,t]$ that $X$ attains a new maximum. Similarly we let $\underline{G}_t:=\sup\{s\in [0,t]:X_s=\underline{X}_s\}$ be, $\PP$-a.s., the last time on $[0,t]$ of attaining the running infimum ($t\geq 0$).
\end{definition}
While the statements of the next proposition are given for the upwards skip-free L\'evy chain $X$, they in fact hold true for the Wiener-Hopf factorization of \emph{any} compound Poisson process. Moreover, they are (essentially) known \cite{kyprianou}. Nevertheless, we begin with these general observations, in order to (a) introduce further relevant notation and (b) provide the reader with the prerequisites needed to understand the remainder of this subsection. Immediately following Proposition~\ref{proposition:WHcompoundPoissonprocess}, however, we particularize to our the skip-free setting. 

\begin{proposition}\label{proposition:WHcompoundPoissonprocess}
Let $p>0$. Then:
\begin{enumerate}[(i)]
\item The pairs $(\overline{G}^*_{e_p},\overline{X}_{e_p})$ and $(e_p-\overline{G}^*_{e_p},\overline{X}_{e_p}-X_{e_p})$ are independent and infinitely divisible, yielding the factorisation: $$\frac{p}{p-i\eta-\Psi(\theta)}=\Psi^+_p(\eta,\theta)\Psi^-_p(\eta,\theta),$$ where for $\{\theta,\eta\}\subset\mathbb{R}$, $$\Psi_p^+(\eta,\theta):=\EE[\exp\{i\eta \overline{G}^*_{e_p}+i\theta \overline{X}_{e_p}\}]\text{ and }\Psi_p^-(\eta,\theta):=\EE[\exp\{i\eta \underline{G}_{e_p}+i\theta \underline{X}_{e_p}\}].$$ Duality: $(e_p-\overline{G}^*_{e_p},\overline{X}_{e_p}-X_{e_p})$ is equal in distribution to $(\underline{G}_{e_p},-\underline{X}_{e_p})$. $\Psi^+_p$ and $\Psi^-_p$ are the Wiener-Hopf factors.
\item\label{theorem:WHCP:ii}  The Wiener-Hopf factors may be identified as follows: $$\EE[\exp\{-\alpha\overline{G}^*_{e_p}-\beta \overline{X}_{e_p}\}]=\frac{\kappa^*(p,0)}{\kappa^*(p+\alpha,\beta)}$$ and $$\EE[\exp\{-\alpha\underline{G}_{e_p}+\beta\underline{X}_{e_p}\}]=\frac{\hat{\kappa}(p,0)}{\hat{\kappa}(p+\alpha,\beta)}$$ for $\{\alpha,\beta\}\subset \Crr$. 
\item\label{theorem:WHCP:iii} Here, in terms of the law of $X$, $$\kappa^*(\alpha,\beta):=k^*\exp\left(\int_0^\infty\int_{(0,\infty)}(e^{-t}-e^{-\alpha t-\beta x})\frac{1}{t}\PP(X_t\in dx)dt\right)$$ and $$\hat{\kappa}(\alpha,\beta)=\hat{k}\exp\left(\int_0^\infty\int_{(-\infty,0]}(e^{-t}-e^{-\alpha t+\beta x})\frac{1}{t}\PP(X_t\in dx)dt\right)$$ for $\alpha\in \Cr$, $\beta\in \Crr$ and some constants $\{k^*,\hat{k}\}\subset\mathbb{R}^+$. 
%\item For some constant $k'<0$ and then all $\theta\in \mathbb{R}$: $$k'\Psi(\theta)=\kappa^*(0,-i\theta)\hat{\kappa}(0,i\theta).$$
\end{enumerate}
\end{proposition}

\begin{proof}
These claims are contained in the remarks regarding compound Poisson processes in \cite[pp. 167-168]{kyprianou} pursuant to the proof of Theorem 6.16 therein. Analytic continuations have been effected in part \ref{theorem:WHCP:iii} using properties of zeros of holomorphic functions \cite[p. 209, Theorem~10.18]{rudin}, the theorems of Cauchy, Morera and Fubini, and finally the finiteness/integrability properties of $q$-potential measures \cite[p. 203, Theorem~30.10(ii)]{sato}.
\end{proof}

\begin{remark}\label{remark:ascending_descennding}
\leavevmode
\begin{enumerate}[(i)]
\item\label{remark:ascending_descennding:i} \cite[pp. 157, 168]{kyprianou} $\hat{\kappa}$ is also the Laplace exponent of the (possibly killed) bivariate descending ladder subordinator $(\hat{L}^{-1},\hat{H})$, where $\hat{L}$ is a local time at the minimum, and the descending ladder heights process $\hat{H}=X_{\hat{L}^{-1}}$ (on $\{\hat{L}^{-1}<\infty\}$; $+\infty$ otherwise) is $X$ sampled at its right-continuous inverse $\hat{L}^{-1}$: $$\EE[e^{-\alpha \hat{L}^{-1}_1-\beta \hat{H}_1}\mathbbm{1}_{\{1<\hat{L}_\infty\}}]=e^{-\hat{\kappa}(\alpha,\beta)},\quad \{\alpha,\beta\}\subset \Crr.$$ 
\item\label{remark:ascending_descennding:ii} As for the strict ascending ladder heights subordinator $H^*:=X_{{L^*}^{-1}}$ (on ${L^*}^{-1}<\infty$; $+\infty$ otherwise), ${L^*}^{-1}$ being the right-continuous inverse of $L^*$, and $L^*$ denoting the amount of time $X$ has spent at a new maximum, we have, thanks to the skip-free property of $X$, as follows. Since $\PP(T_h<\infty)=e^{-\Phi(0)h}$, $X$ stays at a newly achieved maximum each time for an $\Exp(\measure(\RR))$-distributed amount of time, departing it to achieve a new maximum later on with probability $e^{-\Phi(0)h}$, and departing it, never to achieve a new maximum thereafter, with probability $1-e^{-\Phi(0)h}$. It follows that the Laplace exponent of $H^*$ is given by: $$-\log\EE[e^{-\beta H_1}\mathbbm{1}(H_1<+\infty)]=(1-e^{-\beta h})\measure(\RR)e^{-\Phi(0)h}+\measure(\RR)(1-e^{-\Phi(0)h})=\measure(\RR)(1-e^{-(\beta+\Phi(0))h})$$ (where $\beta\in \mathbb{R}_+$). In other words, $H^*/h$ is a killed Poisson process of intensity $\measure(\RR)e^{-\Phi(0)h}$ and with killing rate $\measure(\RR)(1-e^{-\Phi(0)h})$.
\end{enumerate}
\end{remark}
Again thanks to the skip-free nature of $X$, we can expand on the contents of Proposition~\ref{proposition:WHcompoundPoissonprocess}, by offering further details of the Wiener-Hopf factorization. Indeed, if we let $N_t:=\overline{X}_{t}/h$ and $T_k:=T_{kh}$ ($t\geq 0$, $k\in\mathbb{N}_0$) then clearly $T:=(T_k)_{k\geq 0}$ are the arrival times of a renewal process (with a possibly defective inter-arrival time distribution) and $N:=(N_t)_{t\geq 0}$ is the `number of arrivals' process. One also has the relation: $\overline{G}^*_{t}=T_{N_t}$, $t\geq 0$ ($\PP$-a.s.). Thus the random variables entering the Wiener-Hopf factorization are determined in terms of the renewal process $(T,N)$. 

Moreover, we can proceed to calculate explicitly the Wiener-Hopf factors as well as $\hat{\kappa}$ and $\kappa^*$. Let $p>0$. First, since $\overline{X}_{e_p}/h$ is a geometrically distributed random variable, we have, for any $\beta\in\Crr$: 

\begin{equation}\label{eq:wh1}
\EE[e^{-\beta \overline{X}_{e_p}}]=\sum_{k=0}^\infty e^{-\beta hk}(1-e^{-\Phi(p)h})e^{-\Phi(p)hk}=\frac{1-e^{-\Phi(p)h}}{1-e^{-\beta h-\Phi(p)h}}.
\end{equation}
Note here that $\Phi(p)>0$ for all $p>0$. On the other hand, using conditioning (for any $\alpha\geq 0$):

\begin{eqnarray*}
\EE\left[e^{-\alpha\overline{G}^*_{e_p}}\right]&=&\EE\left[\left((u,t)\mapsto \sum_{k=0}^\infty \mathbbm{1}_{[0,\infty)}(t_k) e^{-\alpha t_{k}}\mathbbm{1}_{[t_{k},t_{k+1})}(u)\right)\circ (e_p,T)\right]\\
&=&\EE\left[\left(t\mapsto \sum_{k=0}^\infty \mathbbm{1}_{[0,\infty)}(t_k)e^{-\alpha t_k}(e^{-pt_{k}}-e^{-pt_{k+1}})\right)\circ T\right],\text{ since }e_p\perp T\\
&=&\EE\left[\sum_{k=0}^\infty \mathbbm{1}_{\{T_k<\infty\}}\left(e^{-(p+\alpha)T_k}-e^{-(p+\alpha)T_k}e^{-p(T_{k+1}-T_k)}\right)\right]\\
&=&\EE\left[\sum_{k=0}^\infty e^{-(p+\alpha)T_k}\mathbbm{1}_{\{T_k<\infty\}}\left(1-e^{-p(T_{k+1}-T_k)}\right)\right].
\end{eqnarray*}
Now, conditionally on $T_k<\infty$, $T_{k+1}-T_k$ is independent of $T_k$ and has the same distribution as $T_1$. Therefore, by (\ref{eq:laplace_exponent_for_T_x}) and the theorem of Fubini:

\begin{equation}\label{eq:WH2}
\EE[e^{-\alpha\overline{G}^*_{e_p}}]=\sum_{k=0}^\infty e^{-\Phi(p+\alpha)hk}(1-e^{-\Phi(p)h})=\frac{1-e^{-\Phi(p)h}}{1-e^{-\Phi(p+\alpha)h}}.
\end{equation}
We identify from \eqref{eq:wh1} for any $\beta\in \Crr$: $\frac{\kappa^*(p,0)}{\kappa^*(p,\beta)}=\frac{1-e^{-\Phi(p)h}}{1-e^{-\beta h-\Phi(p)h}}$ and therefore for any $\alpha\geq 0$: $\frac{\kappa^*(p+\alpha,0)}{\kappa^*(p+\alpha,\beta)}=\frac{1-e^{-\Phi(p+\alpha)h}}{1-e^{-\beta h-\Phi(p+\alpha)h}}.$ We identify from \eqref{eq:WH2} for any $\alpha\geq 0$: $\frac{\kappa^*(p,0)}{\kappa^*(p+\alpha,0)}=\frac{1-e^{-h\Phi(p)}}{1-e^{-\Phi(p+\alpha)h}}.$ Therefore, multiplying the last two equalities, for $\alpha\geq 0$ and $\beta\in \Crr$, the equality: 
\begin{equation}\label{eq:WH3}
\frac{\kappa^*(p,0)}{\kappa^*(p+\alpha,\beta)}=\frac{1-e^{-\Phi(p)h}}{1-e^{-\beta h-\Phi(p+\alpha)h}}
\end{equation}
obtains. In particular, for $\alpha>0$ and $\beta\in\Crr$, we recognize for some constant $k^*\in (0,\infty)$: $\kappa^*(\alpha,\beta)=k^*(1-e^{-(\beta+\Phi(\alpha))h})$.
%\footnote{This is not immediately obvious, but can be seen as follows. Let the denominator on the right-hand side of (\ref{eq:WH3}) be $g(p+\alpha,\beta)$ and the numerator $g(p,0)$ for an obvious function $g$. Then $\frac{\kappa^*(p,0)}{g(p,0)}=\frac{\kappa^*(p+\alpha,\beta)}{g(p+\alpha,\beta)}$. Now, for a fixed $\alpha>0$ and $\beta$ the limit as $p\downarrow 0$ on the right-hand side exists ($\kappa^*(\cdot,\beta)$ is continuous on $(0,\infty)$ by the DCT \cite[p. 203, Theorem 30.10(ii)]{sato}) and is some $k^\star\in (0,\infty)$. Therefore the left-hand side limit exists also. It follows that $k^*g(\alpha,\beta)=\kappa^*(\alpha,\beta)$ for every $\alpha>0$ and every $\beta\in\Crr$.} 
Next, observe that by independence and duality (for $\alpha\geq 0$ and $\theta\in\mathbb{R}$): 
\begin{eqnarray*}
&&\EE[\exp\{-\alpha\overline{G}^*_{e_p}+i\theta\overline{X}_{e_p}\}]\EE[\exp\{-\alpha\underline{G}_{e_p}+i\theta\underline{X}_{e_p}\}]=\int_0^\infty dtpe^{-pt}\EE[\exp\{-\alpha t+i\theta X_t\}]=\\
&&\int_0^\infty dtpe^{-pt-\alpha t+\Psi(\theta)t}=\frac{p}{p+\alpha-\Psi(\theta)}.
\end{eqnarray*}
Therefore: $$(p+\alpha-\psi(i\theta))\frac{\hat{\kappa}(p,0)}{\hat{\kappa}(p+\alpha,i\theta)}=p\frac{1-e^{i\theta h-\Phi(p+\alpha)h}}{1-e^{-\Phi(p)h}}.$$ Both sides of this equality are continuous in $\theta\in\Cdd$ and analytic in $\theta\in \Cd$. They agree on $\mathbb{R}$, hence agree on $\Cdd$ by analytic continuation. Therefore (for all $\alpha\geq 0$, $\beta\in\Crr$): 
\begin{equation}\label{eq:WH4}
(p+\alpha-\psi(\beta))\frac{\hat{\kappa}(p,0)}{\hat{\kappa}(p+\alpha,\beta)}=p\frac{1-e^{\beta h-\Phi(p+\alpha)h}}{1-e^{-\Phi(p)h}},
\end{equation}
i.e. for all $\beta\in \Crr$ and $\alpha\geq 0$ for which $p+\alpha\neq \psi(\beta)$ one has: $$\EE[\exp\{-\alpha\underline{G}_{e_p}+\beta\underline{X}_{e_p}\}]=\frac{p}{p+\alpha-\psi(\beta)}\frac{1-e^{(\beta -\Phi(p+\alpha))h}}{1-e^{-\Phi(p)h}}.$$ Moreover, for the unique $\beta_0>0$, for which $\psi(\beta_0)=p+\alpha$, one can take the limit $\beta\to\beta_0$ in the above to obtain: $\EE[\exp\{-\alpha\underline{G}_{e_p}+\beta_0\underline{X}_{e_p}\}]=\frac{ph}{\psi'(\beta_0)(1-e^{-\Phi(p)h})}=\frac{ph\Phi'(p+\alpha)}{1-e^{-\Phi(p)h}}$. We also recognize from (\ref{eq:WH4}) for $\alpha>0$ and $\beta\in\Crr$ with $\alpha\ne \psi(\beta)$, and some constant $\hat{k}\in (0,\infty)$:
%\footnote{Again this is not immediately clear, but one can rewrite, with the obvious function $g$, (\ref{eq:WH4}) as $\hat{\kappa}(p,0)g(p,0)=\hat{\kappa}(p+\alpha,\beta)g(p+\alpha,\beta)$, whenever $g$ on the right-hand side is well-defined (i.e. $p+\alpha\ne \psi(\beta)$). Again for any fixed $\alpha>0$ and $\beta\in \Crr$ with $\alpha\ne \psi(\beta)$ the limit as $p\downarrow 0$ on the right-hand side exists and is some $\hat{k}\in (0,\infty)$. Therefore the left-hand side limit exists also and for $\alpha>0$ and $\beta\in \Crr$ with $\alpha\ne \psi(\beta)$, $\hat{k}=\hat{\kappa}(\alpha,\beta)g(\alpha,\beta)$.}
$\hat{\kappa}(\alpha,\beta)=\hat{k}\frac{\alpha-\psi(\beta)}{1-e^{(\beta-\Phi(\alpha))h}}$. With $\beta_0=\Phi(\alpha)$ one can take the limit in the latter as $\beta\to\beta_0$ to obtain: $\hat{\kappa}(\alpha,\beta_0)=\hat{k}\psi'(\beta_0)/h=\frac{\hat{k}}{h\Phi'(\alpha)}$.

In summary: 

\begin{theorem}[Wiener-Hopf factorization for upwards skip-free L\'evy chains]\label{theorem:WienerHopfUSF}
We have the following identities in terms of $\psi$ and $\Phi$: 
\begin{enumerate}[(i)]
\item\label{theorem:USF:WHi} For every $\alpha\geq 0$ and $\beta\in\Crr$: $$\EE[\exp\{-\alpha\overline{G}^*_{e_p}-\beta \overline{X}_{e_p}\}]=\frac{1-e^{-\Phi(p)h}}{1-e^{-(\beta +\Phi(p+\alpha))h}}$$ and $$\EE[\exp\{-\alpha\underline{G}_{e_p}+\beta\underline{X}_{e_p}\}]=\frac{p}{p+\alpha-\psi(\beta)}\frac{1-e^{(\beta -\Phi(p+\alpha))h}}{1-e^{-\Phi(p)h}}$$ (the latter whenever $p+\alpha\neq \psi(\beta)$; for the unique $\beta_0>0$ such that $\psi(\beta_0)=p+\alpha$, i.e. for $\beta_0=\Phi(p+\alpha)$, one has the right-hand side given by $\frac{ph}{\psi'(\beta_0)(1-e^{-\Phi(p)h})}=\frac{ph\Phi'(p+\alpha)}{1-e^{-\Phi(p)h}}$). 
\item\label{theorem:USF:WHii} For some $\{k^*,\hat{k}\}\subset \mathbb{R}^+$ and then for every $\alpha>0$ and $\beta\in\Crr$: $$\kappa^*(\alpha,\beta)=k^*(1-e^{-(\beta+\Phi(\alpha))h})$$ and $$\hat{\kappa}(\alpha,\beta)=\hat{k}\frac{\alpha-\psi(\beta)}{1-e^{(\beta-\Phi(\alpha))h}}$$ (the latter whenever $\alpha\neq \psi(\beta)$; for the unique $\beta_0>0$ such that $\psi(\beta_0)=\alpha$, i.e. for $\beta_0=\Phi(\alpha)$, one has the right-hand side given by $\hat{k}\psi'(\beta_0)/h=\frac{\hat{k}}{h\Phi'(\alpha)}$). 
\end{enumerate}
\end{theorem}
As a consequence of Theorem~\ref{theorem:WienerHopfUSF}\ref{theorem:USF:WHi}, we obtain the formula for the Laplace transform of the running infimum evaluated at an independent exponentially distributed random time:
\begin{equation}\label{eq:running_infimum_ep}
\EE[e^{\beta \underline{X}_{e_p}}]=\frac{p}{p-\psi(\beta)}\frac{1-e^{(\beta-\Phi(p))h}}{1-e^{-\Phi(p)h}}\hspace{0.5cm}(\beta\in \mathbb{R}_+\backslash \{\Phi(p)\})
\end{equation}
(and $\EE[e^{\Phi(p) \underline{X}_{e_p}}]=\frac{p\Phi'(p)h}{1-e^{-\Phi(p)h}}$). In particular, if $\psi'(0+)>0$, then letting $p\downarrow 0$ in \eqref{eq:running_infimum_ep}, one obtains by the DCT: 
\begin{equation}\label{eq:laplace_for_overall_infimum}
\EE[e^{\beta \underline{X}_\infty}]=\frac{e^{\beta h}-1}{\Phi'(0+)h\psi(\beta)}\hspace{0.5cm}(\beta>0).
\end{equation}

We obtain next from Theorem~\ref{theorem:WienerHopfUSF}\ref{theorem:USF:WHii} (recall also Remark~\ref{remark:ascending_descennding}~\ref{remark:ascending_descennding:i}), by letting $\alpha\downarrow 0$ therein, the Laplace exponent $\phi(\beta):=-\log \EE [e^{-\beta \hat{H}_1}\mathbbm{1}(\hat{H}_1<\infty)]$ of the descending ladder heights process $\hat{H}$: 
\begin{equation}\label{eq:subordinator}
\phi(\beta)(e^{\beta h}-e^{\Phi(0)h})=\psi(\beta),\quad \beta\in \mathbb{R}_+,
\end{equation} where we have set for simplicity $\hat{k}=e^{-\Phi(0)h}$, by insisting on a suitable choice of the local time at the minimum. This gives the following characterization of the class of Laplace exponents of the descending ladder heights processes of upwards skip-free L\'evy chains (cf. \cite[Theorem~1]{hubalek}):

\begin{theorem}\label{theorem:subordinator}
Let $h\in (0,\infty)$, $\{\gamma,q\}\subset \mathbb{R}_+$, and $(\phi_k)_{k\in \mathbb{N}}\subset \mathbb{R}_+$, with $q+\sum_{k\in \mathbb{N}}\phi_k\in (0,\infty)$. Then:
\begin{quote}
There exists (in law) an upwards-skip free L\'evy chain $X$ with values in $\Zh$ and with (i) $\gamma$ being the killing rate of its strict ascending ladder heights process (see Remark~\ref{remark:ascending_descennding}\ref{remark:ascending_descennding:ii}), and (ii) $\phi(\beta)=q+\sum_{k=1}^\infty\phi_k(1-e^{-\beta kh})$, $\beta\in \mathbb{R}_+$, being the Laplace exponent of its descending ladder heights process.
\end{quote}
if and only if the following conditions are satisfied: 
\begin{enumerate}[(1)]
\item $\gamma q=0$.
\item\label{subordinator:cond:2} Setting $x$ equal to $1$, when $\gamma=0$, or to the unique solution of the equation: $$\gamma=(1-1/x)\left(\phi_1+x\sum_{k\in \mathbb{N}}\phi_k\right)$$ on the interval $x\in (1,\infty)$, otherwise\footnote{It is part of the condition, that such an $x$ should exist (automatically, given the preceding assumptions, there is at most one).
%Such a solution necessarily exists under the provisions that have already been made; indeed it is equal to $x=\frac{\gamma+\sum_{k\geq 2}\phi_k+\sqrt{(\gamma+\sum_{k\geq 2})^2+4\phi_1\sum_{k~\geq 1}}}{2\sum_{k\geq 1}\phi_k}$.
}; and then defining $\lambda_1:=q+\sum_{k\in \mathbb{N}}\phi_k$, $\lambda_{-k}:=x\phi_k-\phi_{k+1}$, $k\in \mathbb{N}$; it holds: 
$$\lambda_{-k}\geq 0,\quad k\in \mathbb{N}.$$
\end{enumerate}
Such an $X$ is then unique (in law), is called the parent process, its L\'evy measure is given by $\sum_{k\in \mathbb{N}}\lambda_{-k}\delta_{-kh}+\lambda_1\delta_{h}$, and $x=e^{\Phi(0)h}$.
\end{theorem}
\begin{remark}
Condition Theorem~\ref{theorem:subordinator}\ref{subordinator:cond:2} is actually quite explicit. When $\gamma=0$ (equivalently, the parent process does not drift to $-\infty$), it simply says that the sequence $(\phi_k)_{k\in \mathbb{N}}$ should be nonincreasing. In the case when the parent process $X$ drifts to $-\infty$ (equivalently, $\gamma>0$ (hence $q=0$)), we might choose $x\in (1,\infty)$ first, then $(\phi_k)_{k\geq 1}$, and finally $\gamma$.
\end{remark}
\begin{proof}
Note that with $\phi(\beta)=:q+\sum_{k=1}^\infty\phi_k(1-e^{-\beta kh})$, $x:=e^{\Phi(0)h}$, and comparing the respective Fourier components of the left and the right hand-side, \eqref{eq:subordinator} is equivalent to: 
\begin{enumerate}
\item\label{subordinator:1:p} $q+\sum_{k\in \mathbb{N}}\phi_k=\lambda(\{h\})$.
\item $x(q+\sum_{k\in \mathbb{N}}\phi_k)+\phi_1=\measure(\RR)$.
\item\label{subordinator:3:p} $x\phi_k-\phi_{k+1}=\measure(\{-kh\})$, $k\in \mathbb{N}$.
\end{enumerate}
Moreover, the killing rate of the strict ascending ladder heights processes expresses as $\measure(\RR)(1-1/x)$, whereas \eqref{subordinator:1:p} and \eqref{subordinator:3:p} alone, together imply $q+x\sum_{k\in \mathbb{N}}\phi_k+\phi_1=\measure(\RR)$.

Necessity of the conditions. Remark that the strict ascending ladder heights and the descending ladder heights processes cannot simultaneously have a strictly positive killing rate. Everything else is trivial from the above (in particular, we obtain that such an $X$, when it exists, is unique, and has the stipulated L\'evy measure and $\Phi(0)$).

Sufficiency of the conditions. The compound Poisson process $X$ whose L\'evy measure is given by $\measure=\sum_{k\in \mathbb{N}}\lambda_{-k}\delta_{-kh}+\lambda_1\delta_{h}$ (and whose Laplace exponent we shall denote $\psi$, likewise the largest zero of $\psi$ will be denoted $\Phi(0)$) constitutes an upwards skip-free L\'evy chain. Moreover, since $x=1$, unless $q=0$, we obtain either way that $\phi(\beta)(e^{\beta h}-x)=\psi(\beta)$ with $\phi(\beta):=q+\sum_{k=1}^\infty\phi_k(1-e^{-\beta kh})$, $\beta\geq 0$. Substituting in this relation $\beta:=(\log x)/h$, we obtain at once that if $\gamma>0$ (so $q=0$), that then $X$ drifts to $-\infty$, $x=e^{\Phi(0)h}$, and hence $\gamma=(1-e^{-\Phi(0)})\measure(\RR)$ is the killing rate of the strict ascending ladder heights process. On the other hand, when $\gamma=0$, then $x=1$, and a direct computation reveals $\psi'(0+)=h\lambda_1-\sum_{k\in \mathbb{N}}kh(\phi_k-\phi_{k+1})=h(\lambda_1-\sum_{k\in \mathbb{N}}\phi_k)=hq\geq 0$. So $X$ does not drift to $-\infty$, and $\Phi(0)=0$, whence (again) $x=e^{\Phi(0)h}$. Also in this case, the killing rate of the strict ascending ladder heights process is $0=(1-x)\measure(\RR)$. Finally, and regardless of whether $\gamma$ is strictly positive or not, comparing to \eqref{eq:subordinator}, we conclude that $\phi$ is indeed the Laplace exponent of the descending ladder heights process of $X$. 
\end{proof}

\section{Theory of scale functions}\label{section:Theory_of_scale_functions}
Again the reader is invited to compare the exposition of the following section with that of \cite[Section VII.2]{bertoin} and \cite[Section 8.2]{kyprianou}, which deal with the spectrally negative case. 

\subsection{The scale function $\sc$}
It will be convenient to consider in this subsection the times at which $X$ attains a new maximum. We let $D_1$, $D_2$ and so on, denote the depths (possibly zero, or infinity) of the excursions below these new maxima. For $k\in \mathbb{N}$, it is agreed that $D_k=+\infty$ if the process $X$ never reaches the level $(k-1)h$. Then it is clear that for $y\in \Zh^+$, $x\geq 0$ (cf. \cite[p. 137, Paragraph~6.2.4(a)]{buhlman} \cite[Section 9.3]{doney}):
\begin{eqnarray*} 
&&\PP(\underline{X}_{T_{y}}\geq -x)=\PP(D_1\leq x,D_2\leq x+h,\ldots,D_{y/h}\leq x+y-h)=\\
&&\PP(D_1\leq x)\cdot \PP(D_1\leq x+h)\cdots \PP(D_1\leq x+y-h)=\frac{\prod_{r=1}^{\lfloor (y+x)/h\rfloor }\PP(D_1\leq (r-1)h)}{\prod_{r=1}^{\lfloor x/h\rfloor h}\PP(D_1\leq (r-1)h)}=\frac{\sc(x)}{\sc(x+y)},
\end{eqnarray*}
where we have introduced (up to a multiplicative constant) the \emph{scale function}: 
\begin{equation}\label{eq:scale_function}
\sc(x):=1/\prod_{r=1}^{\lfloor x/h\rfloor}\PP(D_1\leq (r-1)h)\hspace{0.5cm}(x\geq 0).
\end{equation}
%Note that $\{\underline{X}_{T_{y}}\geq -x\}$ is the event that $X$ attains $y$ before it ventures into $(-\infty,-x)$, i.e. on leaving $[-x,y)$ the process does so at $y$. 
(When convenient, we extend $\sc$ by $0$ on $(-\infty,0)$.)
\begin{remark}
If needed, we can of course express $\PP(D_1\leq hk)$, $k\in\mathbb{N}_0$, in terms of the usual excursions away from the maximum. Thus, let $\tilde{D}_1$ be the depth of the first excursion away from the current maximum. By the time the process attains a new maximum (that is to say $h$), conditionally on this event, it will make a total of $N$ departures away from the maximum, where (with $J_1$ the first jump time of $X$, $p:=\measure(\{h\})/\measure(\mathbb{R})$, $\tilde{p}:=\PP(X_{J_1}=h\vert T_h<\infty)=p/\PP(T_h<\infty)$) $N\sim \geom(\tilde{p})$. So, denoting $\tilde{\theta}_k:=\PP(\tilde{D}_1\leq hk)$, one has $\PP(D_1\leq hk)=\PP(T_h<\infty) \sum_{l=0}^\infty \tilde{p}(1-\tilde{p})^l\tilde{\theta}_k^l=\frac{p}{1-(1-e^{\Phi(0)h}p)\tilde{\theta}_k}$, $k\in\mathbb{N}_0$. 
%Thus, while it is more convenient to think of excursions from strict new maxima, the connection to excursions away from the running maximum is straight-forward. 
\end{remark}

The following theorem characterizes the scale function in terms of its Laplace transform. 

\begin{theorem}[The scale function]\label{theorem:thescalefnc}
For every $y\in \Zh^+$ and $x\geq 0$ one has: 
\begin{equation}\label{theorem:eq:scalefunction}
\PP(\underline{X}_{T_{y}}\geq -x)=\frac{\sc(x)}{\sc(x+y)}
\end{equation}
and $\sc:[0,\infty)\to [0,\infty)$ is (up to a multiplicative constant) the unique right-continuous and piecewise continuous function of exponential order with Laplace transform: 
\begin{equation}\label{eq:scale_function_laplace}
\hat{\sc}(\beta)=\int_0 ^\infty e^{-\beta x}\sc(x)dx=\frac{e^{\beta h}-1}{\beta h\psi(\beta)}\hspace{0.5cm}(\beta>\Phi(0)).
\end{equation}
\end{theorem}
\begin{proof}
(For uniqueness see e.g. \cite[p. 14, Theorem 10]{engelberg}. It is clear that $\sc$ is of exponential order, simply from the definition (\ref{eq:scale_function}).)

Suppose first $X$ tends to $+\infty$. Then, letting $y\to\infty$ in (\ref{theorem:eq:scalefunction}) above, we obtain $\PP(-\underline{X}_\infty\leq x)=\sc(x)/\sc(+\infty)$. Here, since the left-hand side limit exists by the DCT, is finite and non-zero at least for all large enough $x$, so does the right-hand side, and $\sc(+\infty)\in (0,\infty)$. 

Therefore $\sc(x)=\sc(+\infty)\PP(-\underline{X}_\infty\leq x)$ and hence the Laplace-Stieltjes transform of $\sc$ is given by (\ref{eq:laplace_for_overall_infimum}) --- here we consider $\sc$ as being extended by $0$ on $(-\infty,0)$:

\begin{equation*}
\int_{[0,\infty)} e^{-\beta x}d\sc(x)=\sc(+\infty)\frac{e^{\beta h}-1}{\Phi'(0+)h\psi(\beta)}\hspace{0.5cm}(\beta>0).
\end{equation*}
Since (integration by parts \cite[Chapter 0, Proposition~4.5]{revuzyor}) $\int_{[0,\infty)}e^{-\beta x}d\sc(x)=\beta\int_{(0,\infty)}e^{-\beta x}\sc(x)dx$,
%\footnote{Indeed, this integration by parts formula holds trivially for $\sc=\mathbbm{1}_{[b,\infty)}$, where $b\geq 0$. By linearity it holds for all nondecreasing c\`adl\`ag $\sc$, which consist of finitely many steps. Now, given any nondecrasing, ultimately constant c\`adl\`ag $\sc:[0,\infty)\to[0,\infty)$ one can find a sequence $\sc_n$, pointwise nondecreasing, of such step functions, sharing their limit at $+\infty$ with $\sc$ and which converge to $\sc$ as $n\to\infty$ at all points of continuity of $\sc$. Therefore, by approximation, the left-hand side converges to what we want by the characterization of weak convergence in terms of distribution functions. The right-hand side does the same by MCT. Another application of MCT removes the requirement of eventual constancy (just take $F_n$ as $F$ stopped once it hits level $n\in\mathbb{N}$).} we obtain: 
\begin{equation}\label{eq:laplace_transofrm_drifts_to_inifinity}
\int_{0}^\infty e^{-\beta x}\sc(x)dx=\frac{\sc(+\infty)}{\Phi'(0+)}\frac{e^{\beta h}-1}{\beta h\psi(\beta)}\hspace{0.5cm}(\beta>0).
\end{equation}
Suppose now that $X$ oscillates. Via Remark~\ref{remark:approximate_oscillating_case}, approximate $X$ by the processes $X^\epsilon$, $\epsilon>0$. In \eqref{eq:laplace_transofrm_drifts_to_inifinity}, fix $\beta$, carry over everything except for $\frac{\sc(+\infty)}{\Phi'(0+)}$, divide both sides by $\sc(0)$, and then apply this equality to $X^\epsilon$. Then on the left-hand side, the quantities pertaining to $X^\epsilon$ will converge to the ones for the process $X$ as $\epsilon\downarrow 0$ by the MCT. Indeed, for $y\in \Zh^+$, $\PP(\underline{X}_{T_y}=0)=\sc(0)/\sc(y)$ and (in the obvious notation): $1/\PP(\underline{X^\epsilon}_{T_y^\epsilon}=0)\uparrow 1/\PP(\underline{X}_{T_y}=0)=\sc(y)/\sc(0)$, since $X^\epsilon\downarrow  X$, uniformly on bounded time sets, almost surely as $\epsilon\downarrow 0$. (It is enough to have convergence for $y\in \Zh^+$, as this implies convergence for all $y\geq 0$, $\sc$ being the right-continuous piecewise constant extension of $\sc\vert_{\Zh^+}$.) Thus we obtain in the oscillating case, for some $\alpha\in (0,\infty)$ which is the limit of the right-hand side as $\epsilon\downarrow 0$:
\begin{equation}\label{eq:laplace_transofrm_oscillates}
\int_{0}^\infty e^{-\beta x}\sc(x)dx=\alpha \frac{e^{\beta h}-1}{\beta h\psi(\beta)}\hspace{0.5cm}(\beta>0).
\end{equation}
Finally, we are left with the case when $X$ drifts to $-\infty$. We treat this case by a change of measure (see Proposition~\ref{proposition:exponential_change_of_measure} and the paragraph immediately preceding it). To this end assume, provisionally, that $X$ is already the coordinate process on the canonical filtered space $\mathbb{D}_h$. Then we calculate by Proposition~\ref{proposition:conditioned_to_drift_to_infinity}\ref{lemma:drifttoinfty:two} (for $y\in \Zh^+$, $x\geq 0$):
\begin{eqnarray*}
&&\PP(\underline{X}_{T_y}\geq -x)=\PP(T_y<\infty)\PP(\underline{X}_{T_y}\geq -x\vert T_y<\infty)=e^{-\Phi(0)y}\PP(\underline{X^{T_y}}_{\infty}\geq -x\vert T_y<\infty)=\\
&&e^{-\Phi(0)y}\PPm(\underline{X^{T_y}}_{\infty}\geq -x)=e^{-\Phi(0)y}\PPm(\underline{X}_{T(y)}\geq -x)=e^{-\Phi(0)y}\sc^\natural(x)/\sc^\natural(x+y),
\end{eqnarray*}
where the third equality uses the fact that $(\omega\mapsto \inf\{\omega(s):s\in [0,\infty)\}):(\mathbb{D}_h,\mathcal{F})\to ([-\infty,\infty),\mathcal{B}([-\infty,\infty))$ is a measurable transformation. Here $\sc^\natural$ is the scale function corresponding to $X$ under the measure $\PPm$, with Laplace transform:
\begin{equation*}
\int_{0}^\infty e^{-\beta x}\sc^\natural(x)dx=\frac{e^{\beta h}-1}{\beta h\psi(\Phi(0)+\beta)}\hspace{0.5cm}(\beta>0).
\end{equation*}
Note that the equality $\PP(\underline{X}_{T_y}\geq -x)=e^{-\Phi(0)y}\sc^\natural(x)/\sc^\natural(x+y)$ remains true if we revert back to our original $X$ (no longer assumed to be in its canonical guise). This is so because we can always go from $X$ to its canonical counter-part by taking an image measure. Then the law of the process, hence the Laplace exponent and the probability $\PP(\underline{X}_{T_y}\geq -x)$ do not change in this transformation.

Now define $\tilde{\sc}(x):=e^{\Phi(0)\lfloor1+x/h\rfloor h}\sc^\natural(x)$ ($x\geq 0$). Then $\tilde{\sc}$ is the right-continuous piecewise-constant extension of $\tilde{\sc}\vert_{\Zh^+}$. Moreover, for all $y\in \Zh^+$ and $x\geq 0$, \eqref{theorem:eq:scalefunction} obtains with $\sc$ replaced by $\tilde{\sc}$. Plugging in $x=0$ into \eqref{theorem:eq:scalefunction},  $\tilde{\sc}\vert_{\Zh}$ and  $\sc\vert_{\Zh}$ coincide up to a multiplicative constant, hence $\tilde{\sc}$ and $\sc$ do as well. Moreover, for all $\beta>\Phi(0)$, by the MCT:
\begin{eqnarray*}
\int_0^\infty e^{-\beta x}\tilde{\sc}(x)dx&=&e^{\Phi(0)h}\sum_{k=0}^\infty  \int_{kh}^{(k+1)h}e^{-\beta x}e^{\Phi(0)kh}\sc^\natural(kh)dx\\
&=&e^{\Phi(0)h}\sum_{k=0}^\infty \frac{1}{\beta}e^{-\beta kh}(1-e^{-\beta h})e^{\Phi(0)kh}\sc^\natural(kh)\\
&=&e^{\Phi(0)h}\frac{\beta-\Phi(0)}{\beta}\frac{1-e^{-\beta h}}{1-e^{-(\beta-\Phi(0))h}}\int_0^\infty e^{-(\beta-\Phi(0)) x}\sc^\natural(x)dx\\
&=&e^{\Phi(0)h}\frac{\beta-\Phi(0)}{\beta}\frac{1-e^{-\beta h}}{1-e^{-(\beta-\Phi(0))h}}\frac{e^{(\beta-\Phi(0))h}-1}{(\beta-\Phi(0))h\psi(\beta)}=\frac{(e^{\beta h}-1)}{\beta h\psi(\beta)}.
\end{eqnarray*}
\end{proof}

\begin{remark}
Henceforth the normalization of the scale function $\sc$ will be understood so as to enforce the validity of \eqref{eq:scale_function_laplace}. 
\end{remark}

\begin{proposition}\label{proposition:scale0andINF}
$\sc(0)=1/(h\lambda(\{h\}))$, and $\sc(+\infty)=1/\psi'(0+)$ if $\Phi(0)=0$. If $\Phi(0)>0$, then $\sc(+\infty)=+\infty$.
\end{proposition}
\begin{proof}
Integration by parts and the DCT yield  $\sc(0)=\lim_{\beta\to\infty}\beta \hat{\sc}(\beta)$. \eqref{eq:scale_function_laplace} and another application of the DCT then show that $\sc(0)=1/(h\lambda(\{h\}))$. Similarly, integration by parts and the MCT give the identity $\sc(+\infty)=\lim_{\beta\downarrow 0}\beta\hat{\sc}(\beta)$. The conclusion $\sc(+\infty)=1/\psi'(0+)$ is then immediate from \eqref{eq:scale_function_laplace} when $\Phi(0)=0$. If $\Phi(0)>0$, then the right-hand side of \eqref{eq:scale_function_laplace} tends to infinity as $\beta\downarrow \Phi(0)$ and thus, by the MCT, necessarily $\sc(+\infty)=+\infty$. 
\end{proof}

\subsection{The scale functions $\Wq$, $q\geq 0$}

\begin{definition}\label{definition:qscalefncs}
For $q\geq 0$, let $\sc^{(q)}(x):=e^{\Phi(q)\lfloor 1+ x/h\rfloor h}\sc_{\Phi(q)}(x)$ ($x\geq 0$), where $\sc_c$ plays the role of $\sc$ but for the process $(X,\PP_c)$ ($c\geq 0$; see Proposition~\ref{proposition:exponential_change_of_measure}). Note that $\sc^{(0)}=\sc$. When convenient we extend $\sc^{(q)}$ by $0$ on $(-\infty,0)$. 
\end{definition}

\begin{theorem}\label{theorem:qscalefncs}
For each $q\geq 0$, $\sc^{(q)}:[0,\infty)\to [0,\infty)$ is the unique right-continuous and piecewise continuous function of exponential order with Laplace transform: 
\begin{equation}\label{eq:Wq_Laplace}
\widehat{\Wq}(\beta)=\int_0^\infty e^{-\beta x}W^{(q)}(x)dx=\frac{e^{\beta h}-1}{\beta h(\psi(\beta)-q)}\hspace{0.5cm}(\beta>\Phi(q)).
\end{equation} 
Moreover, for all $y\in \Zh^+$ and $x\geq 0$:
\begin{equation}
\EE[e^{-q T_y}\mathbbm{1}_{\{\underline{X}_{T_y}\geq -x\}}]=\frac{\sc^{(q)}(x)}{\sc^{(q)}(x+y)}.
\end{equation}
\end{theorem}
\begin{proof}
The claim regarding the Laplace transform follows from Proposition~\ref{proposition:exponential_change_of_measure}, Theorem~\ref{theorem:thescalefnc} and Definition~\ref{definition:qscalefncs} as it did in the case of the scale function $\sc$  (cf. final paragraph of the proof of Theorem~\ref{theorem:thescalefnc}). For the second assertion, let us calculate (moving onto the canonical space $\mathbb{D}_h$ as usual, using Proposition~\ref{proposition:exponential_change_of_measure} and noting that $X_{T_y}=y$ on $\{T_y<\infty\}$):
\begin{eqnarray*}
&&\EE[e^{-q T_y}\mathbbm{1}_{\{\underline{X}_{T_y}\geq -x\}}]=\EE[e^{\Phi(q)X_{T_y}-qT_y}\mathbbm{1}_{\{\underline{X}_{T_y}\geq -x\}}]e^{-\Phi(q)y}=\\
&&e^{-\Phi(q)y}\PP_{\Phi(q)}(\underline{X}_{T_y}\geq -x)=e^{-\Phi(q)y}\frac{\sc_{\Phi(q)}(x)}{\sc_{\Phi(q)}(x+y)}=\frac{\sc^{(q)}(x)}{\sc^{(q)}(x+y)}.
\end{eqnarray*}
\end{proof}

\begin{proposition}\label{porosition:Wq0} 
For all $q>0$: $\Wq(0)=1/(h\lambda(\{h\}))$ and $\Wq(+\infty)=+\infty$.
\end{proposition}
\begin{proof}
As in Proposition~\ref{proposition:scale0andINF}, $\Wq(0)=\lim_{\beta\to\infty}\beta \widehat{\Wq}(\beta)=1/(h\lambda(\{h\}))$. Since $\Phi(q)>0$, $\Wq(+\infty)=+\infty$ also follows at once from the expression for $\widehat{\Wq}$. 
\end{proof}
Moreover: 

\begin{proposition}
For $q\geq 0$:
\begin{enumerate}[(i)]
\item If $\Phi(q)>0$ or $\psi'(0+)>0$, then $\lim_{x\to\infty}\Wq(x)e^{-\Phi(q)\lfloor 1+ x/h\rfloor h}=1/\psi'(\Phi(q))$. 
\item If $\Phi(q)=\psi'(0+)=0$ (hence $q=0$), then $\Wq(+\infty)=+\infty$, but $\limsup_{x\to\infty}\Wq(x)/x<\infty$. Indeed, $\lim_{x\to\infty}\Wq(x)/x=2/m_2$, if $m_2:=\int y^2\measure(dy)<\infty$ and $\lim_{x\to\infty}\Wq(x)/x=0$, if $m_2=\infty$.
\end{enumerate}
\end{proposition}
\begin{proof}
The first claim is immediate from Proposition~\ref{proposition:scale0andINF}, Definition~\ref{definition:qscalefncs} and Proposition~\ref{proposition:exponential_change_of_measure}. To handle the second claim, let us calculate, for the Laplace transform $\widehat{dW}$ of the measure $dW$, the quantity (using integration by parts, Theorem~\ref{theorem:thescalefnc} and the fact that (since $\psi'(0+)=0$) $\int y\measure(dy)=0$):
\begin{equation*}
\lim_{\beta\downarrow 0}\beta\widehat{dW}(\beta)=\lim_{\beta\downarrow 0}\frac{\beta^2}{\psi(\beta)}=\frac{2}{m_2}\in [0,+\infty).
\end{equation*}
For: 
\begin{equation*}
\lim_{\beta\downarrow 0}\int (e^{\beta y}-1)\measure(dy)/\beta^2=\lim_{\beta\downarrow 0}\int \frac{e^{\beta y}-\beta y-1}{\beta^2y^2}y^2\measure(dy)=\frac{m_2}{2},
\end{equation*}
by the MCT, since $(u\mapsto \frac{e^{-u}+u-1}{u^2})$ is nonincreasing on $(0,\infty)$ (the latter can be checked by comparing derivatives). The claim then follows by the Karamata Tauberian Theorem \cite[p. 37, Theorem 1.7.1 with $\rho=1$]{bgt}.
\end{proof}
\subsection{The functions $Z^{(q)}$, $q\geq 0$}

\begin{definition}
For each $q\geq 0$, let $\Zq(x):=1+q\int_0^{\lfloor x/h\rfloor h}\Wq(z)dz$ ($x\geq 0$). When convenient we extend these functions by $1$ on $(-\infty,0)$. 
\end{definition}

\begin{definition}
For $x\geq 0$, let $T_x^-:=\inf\{t\geq 0:X_t<-x\}$. 
\end{definition}

\begin{proposition}\label{proposition:a_measure_equality}
In the sense of measures on the real line, for every $q>0$: 
\begin{equation*}
\PP_{-\underline{X}_{e_q}}=\frac{qh}{e^{\Phi(q)h}-1}d\Wq-q\Wq(\cdot-h)\cdot \Delta,
\end{equation*}
where $\Delta:=h\sum_{k=1}^\infty\delta_{kh}$ is the normalized counting measure on $\Zh^{++}\subset \mathbb{R}$, $\PP_{-\underline{X}_{e_q}}$ is the law of $-\underline{X}_{e_q}$ under $\PP$, and $(\Wq(\cdot-h)\cdot \Delta)(A)=\int_A\Wq(y-h)\Delta(dy)$ for Borel subsets $A$ of $\mathbb{R}$. 
\end{proposition}

\begin{theorem}\label{theorem:Zqs}
For each $x\geq 0$, 
\begin{equation}
\EE[e^{-qT_x^-}\mathbbm{1}_{\{T_x^-<\infty\}}]=\Zq(x)-\frac{qh}{e^{\Phi(q)h}-1}\Wq(x)
\end{equation}
when $q>0$, and $\PP(T_x^-<\infty)=1-\sc(x)/\sc(+\infty)$. The Laplace transform of $\Zq$, $q\geq 0$, is given by: 
\begin{equation}\label{eq:Zq_Laplace}
\widehat{\Zq}(\beta)=\int_0^\infty \Zq(x)e^{-\beta x}dx=\frac{1}{\beta}\left(1+\frac{q}{\psi(\beta)-q}\right), \hspace{0.5cm} (\beta>\Phi(q)).
\end{equation}
\end{theorem}
\noindent\emph{Proof of Proposition~\ref{proposition:a_measure_equality} and Theorem~\ref{theorem:Zqs}}. 
First, with regard to the Laplace transform of $\Zq$, we have the following derivation (using integration by parts, for every $\beta>\Phi(q)$):
\begin{eqnarray*}
\int_0^\infty\Zq(x)e^{-\beta x}dx&=&\int_0^\infty\frac{e^{-\beta x}}{\beta}d\Zq(x)=\frac{1}{\beta}\left(1+q\sum_{k=1}^\infty e^{-\beta kh}\Wq((k-1)h)h\right)\\
&=&\frac{1}{\beta}\left(1+\frac{qe^{-\beta h}\beta h}{1-e^{-\beta h}}\sum_{k=1}^\infty \frac{(1-e^{-\beta h})}{\beta}e^{-\beta (k-1)h}\Wq((k-1)h)\right)\\
&=&\frac{1}{\beta}\left(1+q\frac{\beta h}{e^{\beta h}-1}\widehat{\Wq}(\beta)\right)=\frac{1}{\beta}\left(1+\frac{q}{\psi(\beta)-q}\right).
\end{eqnarray*}

Next, to prove Proposition~\ref{proposition:a_measure_equality}, note that it will be sufficient to check the equality of the Laplace transforms \cite[p. 109, Theorem 8.4]{bhattacharya}. By what we have just shown, \eqref{eq:running_infimum_ep}, integration by parts, and Theorem~\ref{theorem:qscalefncs}, we need then only establish, for $\beta>\Phi(q)$:
\begin{equation*}
\frac{q}{\psi(\beta)-q}\frac{e^{(\beta-\Phi(q))h}-1}{1-e^{-\Phi(q)h}}=\frac{qh}{e^{\Phi(q)h}-1}\frac{\beta(e^{\beta h}-1)}{(\psi(\beta)-q)\beta h}-\frac{q}{\psi(\beta)-q},
\end{equation*}
which is clear. 

Finally, let $x\in\Zh^+$. For $q>0$, evaluate the measures in Proposition~\ref{proposition:a_measure_equality} at $[0,x]$, to obtain:
\begin{eqnarray*}
\EE[e^{-qT_x^-}\mathbbm{1}_{\{T_x^-<\infty\}}]&=&\PP(e_q\geq T_x^-)=\PP(\underline{X}_{e_q}<-x)=1-\PP(\underline{X}_{e_q}\geq -x)\\
&=&1+q\int_0^x\Wq(z)dz-\frac{qh}{e^{\Phi(q)h}-1}\Wq(x),
\end{eqnarray*}
whence the claim follows. On the other hand, when $q=0$, the following calculation is straightforward: $\PP(T_x^-<\infty)=\PP(\underline{X}_\infty<-x)=1-\PP(\underline{X}_\infty\geq -x)=1-\sc(x)/\sc(+\infty)$ (we have passed to the limit $y\to\infty$ in \eqref{theorem:eq:scalefunction} and used the DCT on the left-hand side of this equality).\qed

\begin{proposition}
Let $q\geq 0$, $x\geq 0$, $y\in \Zh^+$. Then: $$\EE[e^{-qT_x^-}\mathbbm{1}_{\{T_x^-<T_y\}}]=\Zq(x)-\Zq(x+y)\frac{\Wq(x)}{\Wq(x+y)}.$$
\end{proposition}
\begin{proof}
Observe that $\{T_x^-=T_y\}=\emptyset$, $\PP$-a.s. The case when $q=0$ is immediate and indeed contained in Theorem~\ref{theorem:thescalefnc}, since, $\PP$-a.s., $\Omega\backslash \{T_x^-<T_y\}=\{T_x^-\geq T_y\}=\{\underline{X}_{T_y}\geq -x\}$. For $q>0$ we observe that by the strong Markov property, Theorem~\ref{theorem:qscalefncs} and Theorem~\ref{theorem:Zqs}:
\begin{eqnarray*}
&&\EE[e^{-qT_x^-}\mathbbm{1}_{\{T_x^-<T_y\}}]=\EE[e^{-qT_x^-}\mathbbm{1}_{\{T_x^-<\infty\}}]-\EE[e^{-qT_x^-}\mathbbm{1}_{\{T_y<T_x^-<\infty\}}]\\
&=&\Zq(x)-\frac{qh}{e^{\Phi(q)h}-1}\Wq(x)-\EE[e^{-qT_y}\mathbbm{1}_{\{T_y<T_x^-\}}]\EE[e^{-qT_{x+y}^-}\mathbbm{1}_{\{T_{x+y}^-<\infty\}}]\\
&=&\Zq(x)-\frac{qh}{e^{\Phi(q)h}-1}\Wq(x)-\frac{\Wq(x)}{\Wq(x+y)}\left(\Zq(x+y)-\frac{qh}{e^{\Phi(q)h}-1}\Wq(x+y)\right)\\
&=&\Zq(x)-\Zq(x+y)\frac{\Wq(x)}{\Wq(x+y)}.
\end{eqnarray*}

\end{proof}

\subsection{Calculating scale functions}
In this subsection it will be assumed for notational convenience, but without loss of generality, that $h=1$. %and that $X$ is the canonical process on $\Omega=\mathbb{D}_h$ equipped with the usual $\sigma$-algebra and filtration. 
We define: 
$$\gamma:= \measure(\mathbb{R}),\quad p:=\measure(\{1\})/\gamma,\quad q_k:=\measure(\{-k\})/\gamma,\ k\geq 1.$$
Fix $q\geq 0$. Then denote, provisionally, $e_{m,k}:=\EE[e ^{-q T_k}\mathbbm{1}_{\{\underline{X}_{T_k}\geq -m\}}]$, and $e_k:=e_{0,k}$, where $\{m,k\}\subset \mathbb{N}_0$ and note that, thanks to Theorem~\ref{theorem:qscalefncs},  $e_{m,k}=\frac{e_{m+k}}{e_m}$ for all $\{m,k\}\subset \mathbb{N}_0$. Now, $e_0=1$. Moreover, by the strong Markov property, for each $k\in\mathbb{N}_0$, by conditioning on $\mathcal{F}_{T_k}$ and then on $\mathcal{F}_J$, where $J$ is the time of the first jump after $T_k$ (so that, conditionally on $T_k<\infty$, $J-T_k\sim \Exp(\gamma)$): 
\begin{eqnarray*}
e_{k+1}&=&\EE\Big[e^{-q T_k} \mathbbm{1}_{\{\underline{X}_{T_k}\geq 0\}}e^{-q(J-T_k)}\big(\mathbbm{1}(\text{next jump after }T_k\text{ up})+\\
&& \mathbbm{1}(\text{next jump after }T_k\text{ }1\text{ down, then up }2\text{ before down more than }k-1)+\cdots +\\
&&\mathbbm{1}(\text{next jump after }T_k\text{ }k\text{ down \& then up }k+1\text{ before down more than 0})\big)e^{-q(T_{k+1}-J)}\Big]\\
&=&e_k\frac{\gamma}{\gamma+q}[p+q_1e_{k-1,2}+\cdots +q_ke_{0,k+1}]=e_k\frac{\gamma}{\gamma+q}[p+q_1\frac{e_{k+1}}{e_{k-1}}+\cdots+q_k\frac{e_{k+1}}{e_0}].
\end{eqnarray*}
Upon division by $e_ke_{k+1}$, we obtain: 
\begin{equation*}
\Wq(k)=\frac{\gamma}{\gamma+q}[p\Wq(k+1)+q_1\Wq(k-1)+\cdots+q_k\Wq(0)].
\end{equation*}
Put another way, for all $k\in\mathbb{Z}_+$:
\begin{equation}\label{equation:recursion:Wq}
p\Wq(k+1)=\left(1+\frac{q}{\gamma}\right)\Wq(k)-\sum_{l=1}^kq_l\Wq(k-l).
\end{equation}
Coupled with the initial condition $\Wq(0)=1/(\gamma p)$ (from Proposition~\ref{porosition:Wq0} and Proposition~\ref{proposition:scale0andINF}), this is an explicit recursion scheme by which the values of $\Wq$ obtain (cf. \cite[Section~4, Equations~(6) \& (7)]{vylder} \cite[Section~7, Equations~(7.1) \& (7.5)]{dickson} \cite[p. 255, Proposition~3.1]{marchal}). We can also see the vector $\Wq=(\Wq(k))_{k\in\ZZ}$ as a suitable eigenvector of the transition matrix $P$ associated to the jump chain of $X$. Namely, we have for all $k\in\ZZ_+$: $\left(1+\frac{q}{\gamma}\right)\Wq(k)=\sum_{l\in\ZZ}P_{kl}\Wq(l)$.

Now, with regard to the function $\Zq$, its values can be computed directly from the values of $\Wq$ by a straightforward summation, $\Zq(n)=1+q\sum_{k=0}^{n-1}\Wq(k)$ ($n\in\mathbb{N}_0$). Alternatively, \eqref{equation:recursion:Wq} yields immediately its analogue, valid for each $n\in\ZZ^+$ (make a summation $\sum_{k=0}^{n-1}$ and multiply by $q$, using Fubini's theorem for the last sum):
\begin{equation*}
p\Zq(n+1)-p-p q\Wq(0)=\left(1+\frac{q}{\gamma}\right)(\Zq(n)-1)-\sum_{l=1}^{n-1}q_l(\Zq(n-l)-1),
\end{equation*}
i.e. for all $k\in\ZZ_+$:
\begin{equation}\label{equation:recursion:Zq}
p\Zq(k+1)+\left(1-p-\!\!\sum_{l=1}^{k-1}q_l\right)=\left(1+\frac{q}{\gamma}\right)\Zq(k)-\!\!\sum_{l=1}^{k-1}q_l\Zq(k-l).
\end{equation}
Again this can be seen as an eigenvalue problem. Namely, for all $k\in \ZZ_+$: $\left(1+\frac{q}{\gamma}\right)\Zq(k)=\sum_{l\in\ZZ}P_{kl}\Zq(l)$. In summary:

\begin{proposition}[Calculation of $\Wq$ and $\Zq$]\label{proposition:calculation}
Let $h=1$ and $q\geq 0$. Seen as vectors, $\Wq:=(\Wq(k))_{k\in\ZZ}$ and $\Zq:=(\Zq(k))_{k\in\ZZ}$ satisfy, entry-by-entry ($P$ being the transition matrix associated to the jump chain of $X$; $\lambda_q:=1+q/\lambda(\mathbb{R})$): \begin{equation}\label{eq:eigenvectors}
(P\Wq)\vert_{\ZZ_+}=\lambda_q\Wq\vert_{\ZZ_+}\text{ and }(P\Zq)\vert_{\ZZ_+}=\lambda_q\Zq\vert_{\ZZ_+},
\end{equation}
i.e. \eqref{equation:recursion:Wq} and \eqref{equation:recursion:Zq} hold true for $k\in \mathbb{Z}_+$. Additionally, $\Wq\vert_{\mathbb{Z}^-}=0$ with $\Wq(0)=1/\lambda(\{1\})$, whereas $\Zq\vert_{\mathbb{Z}_-}=1$. 
\end{proposition}
An alternative form of recursions \eqref{equation:recursion:Wq} and \eqref{equation:recursion:Zq} is as follows:

\begin{corollary}\label{proposition:calculating_scale_functions}
We have for all $n\in \mathbb{N}_0$:
\begin{equation}\label{equation:recursion:Wq:further}
\Wq(n+1)=\Wq(0)+\sum_{k=1} ^{n+1}\Wq(n+1-k)\frac{q+\lambda(-\infty,-k]}{\lambda(\{1\})},\quad \Wq(0)=1/\lambda(\{1\}),
\end{equation}
and for $\widetilde{\Zq}:=\Zq-1$,
\begin{equation}\label{equation:recursion:Zq:further}
\widetilde{\Zq}(n+1)=(n+1)\frac{q}{\lambda\{1\}}+\sum_{k=1}^n\widetilde{\Zq}(n+1-k)\frac{q+\lambda(-\infty,-k]}{\lambda(\{1\})},\quad \widetilde{\Zq}(0)=0.
\end{equation}
\end{corollary}
\begin{proof}
Recursion~\eqref{equation:recursion:Wq:further} obtains from \eqref{equation:recursion:Wq} as follows (cf. also \cite[(proof of) Proposition~XVI.1.2]{asmussen}): 
\footnotesize
\begin{eqnarray*}
&&\!\!\!\!\!\!\!\!p\Wq(n+1)+\sum_{k=1}^nq_k\Wq(n-k)=\nu_q\Wq(n),\forall n\in \mathbb{N}_0\Rightarrow\\
&&\!\!\!\!\!\!\!\!p\Wq(k+1)+\sum_{m=0}^{k-1}q_{k-m}\Wq(m)=\nu_q\Wq(k),\forall k\in \mathbb{N}_0\Rightarrow \text{ (making a summation  }\sum_{k=0}^n\text{)}\\
&&\!\!\!\!\!\!\!\!p\sum_{k=0}^n\Wq(k+1)+\sum_{k=0}^n\sum_{m=0}^{k-1}q_{k-m}\Wq(m)=\nu_q\sum_{k=0}^n\Wq(k),\forall n\in \mathbb{N}_0\Rightarrow\text{ (Fubini)}\\
&&\!\!\!\!\!\!\!\!p\Wq(n+1)+p\sum_{k=0}^n\Wq(k)+\sum_{m=0}^{n-1}\Wq(m)\sum_{k=m+1}^{n}q_{k-m}=p\Wq(0)+\nu_q\sum_{k=0}^n\Wq(k),\forall n\in \mathbb{N}_0\Rightarrow \text{ (relabeling)}\\
&&\!\!\!\!\!\!\!\!p\Wq(n+1)+p\sum_{k=0}^n\Wq(k)+\sum_{k=0}^{n-1}\Wq(k)\sum_{l=1}^{n-k}q_{l}=p\Wq(0)+(1+q/\gamma)\sum_{k=0}^n\Wq(k),\forall n\in \mathbb{N}_0\Rightarrow\text{ (rearranging)}\\
&&\!\!\!\!\!\!\!\!\Wq(n+1)=\Wq(0)+\sum_{k=0}^n\Wq(k)\frac{q+\gamma\sum_{l=n-k+1}^{\infty}q_l}{p\gamma}, \forall n\in \mathbb{N}_0\Rightarrow\text{ (relabeling)}\\
&&\!\!\!\!\!\!\!\!\Wq(n+1)=\Wq(0)+\sum_{k=1}^{n+1}\Wq(n+1-k)\frac{q+\gamma\sum_{l=k}^{\infty}q_l}{p\gamma}, \forall n\in \mathbb{N}_0.
\end{eqnarray*}\normalsize
Then \eqref{equation:recursion:Zq:further} follows from \eqref{equation:recursion:Wq:further} by another summation from $n=0$ to $n=w-1$, $w\in \mathbb{N}_0$, say, and an interchange in the order of summation for the final sum. 
\end{proof}
Now, given these explicit recursions for the calculation of the scale functions, searching for those Laplace exponents of upwards skip-free L\'evy chains (equivalently, their descending ladder heights processes, cf. Theorem~\ref{theorem:subordinator}), that allow for an inversion of \eqref{eq:Wq_Laplace} %-\eqref{eq:Zq_Laplace}  
in terms of some or another (more or less exotic) \emph{special function}, appears less important. This is in contrast to the spectrally negative case, see e.g. \cite{hubalek}. 

That said, when the scale function(s) can be expressed in terms of \emph{elementary functions}, this is certainly note-worthy. In particular, whenever the support of $\measure$ is bounded from below, then \eqref{equation:recursion:Wq} becomes a homogeneous linear difference equation with constant coefficients of some (finite) order, which can always be solved for explicitly in terms of elementary functions (as long as one has control over the zeros of the characteristic polynomial). The minimal example of this situation is of course when $X$ is skip-free to the left also. For simplicity let us only consider the case $q=0$.
\begin{itemize}
\item \textbf{Skip-free chain}. Fix $p\in (0,1]$, let $\measure=p\delta_1+(1-p)\delta_{-1}$. Then $\sc(k)=\frac{1}{1-2p}\left[\left(\frac{1-p}{p}\right)^{k+1}-1\right]$, unless $p=1/2$, in which case $\sc(k)=2(1+k)$, $k\in \mathbb{N}_0$. 
\end{itemize}

Indeed, if one wanted to, one could in general \emph{reverse-engineer} the L\'evy measure, so that the zeros of the characteristic polynomial of \eqref{equation:recursion:Wq} (with $q=0$) are known \emph{a priori}, as follows. Choose $l\in \mathbb{N}$ as being $-\inf\supp(\measure)$; $p\in (0,1)$ as representing the probability of an up-jump; and then the numbers $\lambda_1$, \ldots, $\lambda_{l+1}$ (real, or not), in such a way that the polynomial (in $x$) $p(x-\lambda_1)\cdots (x-\lambda_{l+1})$ coincides with the characteristic polynomial of  \eqref{equation:recursion:Wq} (for $q=0$): $$px^{l+1}-x^l+q_1x^{l-1}+ \cdots + q_l$$ of \emph{some} upwards skip-free L\'evy chain, which can jump down by at most (and does jump down by) $l$ units (this imposes some set of algebraic restrictions on the elements of $\{\lambda_1,\ldots,\lambda_{l+1}\}$). \emph{A priori} one then has access to the zeros of the characteristic polynomial, and it remains to use the linear recursion in order to determine the first $l+1$ values of $\sc$, thereby finding (via solving a set of linear equations of dimension $l+1$) the sought-after particular solution of  \eqref{equation:recursion:Wq} (with $q=0$), that is $\sc$. A particular parameter set for the zeros is depicted in Figure~\ref{figure:parameter_set_zeros}. 
\begin{figure}
                \includegraphics[width=0.76\textwidth]{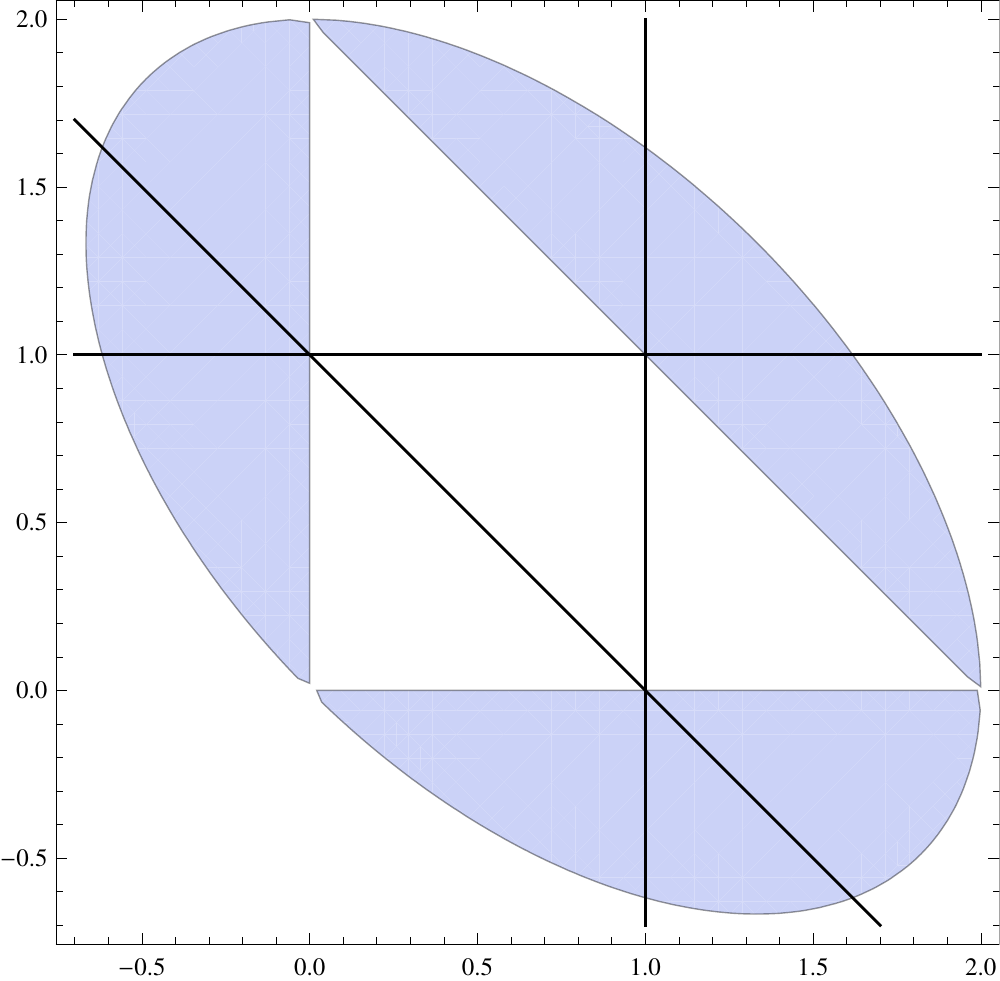}
\caption{Consider the possible zeros $\lambda_1$, $\lambda_2$ and $\lambda_3$ of the characteristic polynomial of \eqref{equation:recursion:Wq} (with $q=0$), when $l:=-\inf \supp(\measure)=2$ and $p=1/2$. Straightforward computation shows they are precisely those that satisfy (o) $\lambda_3=2-\lambda_1-\lambda_2$; (i) $(\lambda_1-1)(\lambda_2-1)(\lambda_1+\lambda_2-1)=0$ and (ii) $\lambda_1\lambda_2+(\lambda_1+\lambda_2)(2-\lambda_1-\lambda_2)\geq 0$ \& $\lambda_1\lambda_2(2-\lambda_1-\lambda_2)<0$. In the plot one has $\lambda_1$ as the abscissa, $\lambda_2$ as the ordinate. The shaded area (an ellipse missing the closed inner triangle) satisfies (ii), the black lines verify (i). Then $q_1=(\lambda_1\lambda_2+(\lambda_1+\lambda_2)(2-\lambda_1-\lambda_2))/2$ and $q_2=(-\lambda_1\lambda_2(2-\lambda_1-\lambda_2))/2$.}\label{figure:parameter_set_zeros}
\end{figure} 

An example in which the support of $\measure$ is not bounded, but one can/could still obtain closed form expressions in terms of elementary functions, is the following.

\begin{itemize}
\item Take $p\in (0,1)$, $q_l=(1-p)(1-a)a^{l-1}$, $l\in \mathbb{N}$, $a\in (0,1)$. Then \eqref{equation:recursion:Wq} implies for $z(k):=W(k)/a^k$: $paz(k+1)=z(k)-\sum_{l=1}^k(1-p)(1-a)z(k-l)/a$, i.e. for $\gamma(k):=\sum_{l=0}^kz(l)$, $pa^2\gamma(k+1)-(a+pa^2)\gamma(k)+(1-p+pa)\gamma(k-1)=0$, a homogeneous second order linear difference equation with constant coefficients.
\end{itemize}
We close with the following remark and corollary (cf. \cite[Equation~(12)]{biffis} and \cite[Remark~5]{avram_pistosrious}, respectively, for their spectrally negative analogues): for them we no longer assume that $h=1$.% or, indeed, that the underlying filtered probability space is the canonical one, i.e. we revert back to our original setting. 

\begin{remark}
Let $L$ be the infinitesimal generator \cite[p. 208, Theorem~31.5]{sato} of $X$. It is seen from \eqref{eq:eigenvectors}, that for each $q\geq 0$, $((L-q)\Wq)\vert_{\mathbb{R}_+}=((L-q)\Zq)\vert_{\mathbb{R}_+}=0$. 
\end{remark}

\begin{corollary}
For each $q\geq 0$, the stopped processes $Y$ and $Z$, defined by $Y_t:=e^{-q(t\land T_0^-)}\Wq\circ X_{t\land T_0^-}$ and $Z_t:=e^{-q(t\land T_0^-)}\Wq\circ X_{t\land T_0^-}$, $t\geq 0$, are nonnegative $\PP$-martingales with respect to the natural filtration $\mathbb{F}^X=(\mathcal{F}^X_s)_{s\geq 0}$ of $X$. 
\end{corollary}
\begin{proof}
We argue for the case of the process $Y$, the justification for $Z$ being similar. Let $(H_k)_{k\geq 1}$, $H_0:=0$, be the sequence of jump times of $X$ (where, possibly by discarding a $\PP$-negligible set, we may insist on all of the $T_k$, $k\in \mathbb{N}_0$, being finite and increasing to $+\infty$ as $k\to\infty$). Let $0\leq s<t$, $A\in \mathcal{F}_s^X$. By the MCT it will be sufficient to establish for $\{l,k\}\subset \mathbb{N}_0$, $l\leq k$, that: 
\footnotesize
\begin{equation}\label{eq:WQZQmtgs}
\EE[\mathbbm{1}(H_l\leq s<H_{l+1})\mathbbm{1}_AY_t\mathbbm{1}(H_k\leq t<H_{k+1})]=\EE[\mathbbm{1}(H_l\leq s<H_{l+1})\mathbbm{1}_AY_s\mathbbm{1}(H_k\leq t<H_{k+1})].
\end{equation}
\normalsize
On the left-hand (respectively right-hand) side of \eqref{eq:WQZQmtgs} we may now replace $Y_t$ (respectively $Y_s$) by $Y_{H_k}$ (respectively $Y_{H_l}$) and then harmlessly insist on $l<k$. Moreover, up to a completion, $\mathcal{F}_s^X\subset \sigma((H_m\land s,X(H_m\land s))_{m\geq 0})$. Therefore, by a $\pi$/$\lambda$-argument, we need only verify \eqref{eq:WQZQmtgs} for sets $A$ of the form: $A=\bigcap_{m=1}^{M}\{H_m\land s\in A_m\}\cap\{X(H_m\land s)\in B_m\}$, $A_m$, $B_m$ Borel subsets of $\mathbb{R}$, $1\leq m\leq M$, $M\in \mathbb{N}$. Due to the presence of the indicator $\mathbbm{1}(H_l\leq s<H_{l+1})$, we may also take, without loss of generality, $M=l$ and hence $A\in \mathcal{F}^X_{H_l}$. Further, $\mathcal{H}:=\sigma(H_{l+1}-H_l,H_k-H_l,H_{k+1}-H_l)$ is independent of $\mathcal{F}^X_{H_l}\lor \sigma(Y_{H_k})$ and then $\EE[Y_{H_k}\vert \mathcal{F}^X_{H_l}\lor \mathcal{H}]=\EE[Y_{H_k}\vert \mathcal{F}^X_{H_l}]=Y_{H_l}$, $\PP$-a.s. (as follows at once from \eqref{eq:eigenvectors} of Proposition~\ref{proposition:calculation}), whence \eqref{eq:WQZQmtgs} obtains.
\end{proof}

\bibliographystyle{plain}
\bibliography{Biblio_fluctuation_theory}

\end{document}